\documentclass[11pt]{article}

\usepackage{amsmath, amsthm, amssymb, tikz, hyperref, enumerate}
\usepackage{amsfonts,shuffle}
\usepackage{fullpage}
\usepackage[enableskew,vcentermath]{youngtab}
\usepackage{multirow}

\newcommand\beq{\begin{equation}}
\newcommand\eeq{\end{equation}}
\newcommand\bce{\begin{center}}
\newcommand\ece{\end{center}}
\newcommand\bea{\begin{eqnarray}}
\newcommand\eea{\end{eqnarray}}
\newcommand\ba{\begin{array}}
\newcommand\ea{\end{array}}
\newcommand\ben{\begin{enumerate}}
\newcommand\een{\end{enumerate}}
\newcommand\bit{\begin{itemize}}
\newcommand\eit{\end{itemize}}
\newcommand\brr{\begin{array}}
\newcommand\err{\end{array}}
\newcommand\bt{\begin{tabular}}
\newcommand\et{\end{tabular}}

\renewcommand\S{{\mathcal S}}

\newcommand\A{{\mathcal A}}
\newcommand\LL{{\mathcal L}}

\newcommand\G{{\mathcal G}}
\newcommand\LLL{{\mathcal Y}}
\newcommand\Sh{\G^{+^2}}

\DeclareMathOperator\SYT{SYT}

\DeclareMathOperator\st{st}

\DeclareMathOperator\shape{shape}

\newcommand{\bbz}{\mathbb{Z}}

\DeclareMathOperator\Des{Des}
\DeclareMathOperator\des{des}
\DeclareMathOperator\cdes{cdes}
\DeclareMathOperator\cDes{cDes}
\newcommand\uu{\mathbf{u}}
\newcommand\vv{\mathbf{v}}
\newcommand{\ww}{\mathbf{w}}

\newcommand{\ZZ}{\mathbb{Z}}

\newcommand{\CC}{\mathbb{C}}

\newcommand{\BBB}{{\mathcal{B}}}
\newcommand{\Comp}{\operatorname{Comp}}
\newcommand{\ch}{\operatorname{ch}}

\newcommand{\Q}{{\mathcal Q}}

\theoremstyle{plain}
\newtheorem{theorem}{Theorem}[section]
\newtheorem{proposition}[theorem]{Proposition}
\newtheorem{lemma}[theorem]{Lemma}
\newtheorem{corollary}[theorem]{Corollary}
\newtheorem{conjecture}[theorem]{Conjecture}
\newtheorem{question}[theorem]{Question}
\newtheorem{problem}[theorem]{Problem}

\theoremstyle{definition}
\newtheorem{defn}[theorem]{Definition}
\newtheorem*{ex}{Example}

\theoremstyle{remark}

\newtheorem{remark}[theorem]{Remark}
\newtheorem{observation}[theorem]{Observation}

\numberwithin{figure}{section}

\title{Schur-positive sets of permutations via products and grid classes}

\author{Sergi Elizalde~\thanks{Department of Mathematics, Dartmouth College, Hanover, NH 03755, USA. {\tt sergi.elizalde@dartmouth.edu}.
} \and Yuval
Roichman~\thanks{Department of Mathematics, Bar-Ilan University,
 Ramat-Gan 52900, Israel.  {\tt yuvalr@math.biu.ac.il}.}}

\date{}

\begin{document}

\maketitle

\begin{abstract}
Characterizing sets of permutations whose
associated quasisymmetric function is symmetric and Schur-positive
is a long-standing problem in algebraic combinatorics.
In this paper we present a general method to construct Schur-positive sets and multisets,
based on geometric grid classes and the product operation. Our approach produces
many new instances of Schur-positive sets, and provides a broad framework
that explains the existence of known such sets that until now were sporadic cases.
\end{abstract}



\tableofcontents

\section{Introduction}

\subsection{Background}

Given any subset $A$ of the symmetric group $\S_n$, define the
quasisymmetric function
\[
\Q(A) := \sum\limits_{\pi\in A} F_{n,\Des(\pi)},
\]
where $\Des(\pi):=\{i:\ \pi(i)>\pi(i+1)\}$ is the descent set of
$\pi$ and $F_{n,\Des(\pi)}$ is Gessel's {\em fundamental
quasisymmetric function} (see Section~\ref{sec:prel_quasi} for
more background and definitions). The quasisymmetric function
$\Q(A)$ was introduced in~\cite{Gessel}, where Gessel was
concerned with the case when $A$ is the set of linear extensions
of a labeled poset\footnote{Gessel was motivated by a well-known
conjecture 
of Stanley~\cite[III, Ch. 21]{Stanley_thesis}, which he
reformulates as follows:
if $A$ is the set of linear extensions of a labeled
poset $P$, then $\Q(A)$ is
symmetric if and only if $P$ is isomorphic to the poset determined
by a skew semistandard Young tableau.}.
The following long-standing problem was first posed
in~\cite{Gessel-Reutenauer}.

\begin{problem}\label{prob:symmetric}
For which subsets of permutations $A\subseteq \S_n$ is $\Q(A)$ symmetric?
\end{problem}

A symmetric function is called {\it Schur-positive} if all
coefficients in its expansion in the Schur basis are nonnegative.
The problem of determining whether a given symmetric function is
Schur-positive is a major problem in contemporary algebraic
combinatorics~\cite{Stanley_problems}.

By analogy, a set (or, more generally, a multiset) $A$ of
permutations in $\S_n$ is called {\it Schur-positive} if $\Q(A)$
is symmetric and Schur-positive. 
Classical examples of Schur-positive sets of permutations include
inverse descent classes and Knuth classes~\cite{Gessel}, conjugacy
classes~\cite[Theorem 5.5]{Gessel-Reutenauer} and permutations of
fixed inversion number~\cite[Prop. 9.5]{Adin-R}.

An exotic example of a Schur-positive
set, different from all the above ones, was recently found: the set of arc permutations, which
may be characterized as those
avoiding the patterns $\{\sigma\in \S_4:\
|\sigma(1)-\sigma(2)|=2\}$. A bijective proof of its Schur-positivity is given in~\cite{ER1}.
Inspired by this example, Woo and Sagan raised the
problem of finding other Schur-positive pattern-avoiding
sets~\cite{Sagan_talk}. Our goal in this paper is to provide a conceptual approach that
explains the existing results and produces new
examples of Schur-positive pattern-avoiding sets of permutations.
An important tool in our approach will be to consider products of geometric
grid classes.

A {\it geometric grid class} consists of those permutations that
can be drawn on a specified set of line segments of slope $\pm1$,
whose locations are determined by the positions of the
corresponding entries in a matrix $M$ with entries in
$\{0,1,-1\}$. An example of a matrix and its corresponding line segments is given in Figure~\ref{fig:grid}.
Let $\G_n(M)$ be the set of permutations in
$\S_n$ that can be obtained by placing $n$ dots on the segments in
such a way that there are no two dots on the same vertical or
horizontal line, labeling the dots with $1,2,\dots,n$ from bottom to top,
and then reading them from left to right.
Geometric grid classes may be characterized by a finite set of
forbidden patterns~\cite{AABRV}. It follows from \cite[Theorem
5.5]{Gessel-Reutenauer} together with elementary combinatorial
arguments that all one-column grid classes, as well as layered and
colayered permutation classes are Schur-positive; see
Section~\ref{sec:elementary} below.

\begin{figure}[h]
\centering
\begin{tikzpicture}[scale=0.8]
\draw (-3.5,1.5) node {$M=\left(
\begin{array}{cc}
0&1\\ -1&0 \\ 1& -1
\end{array}\right)
$};
\draw (0,0) rectangle (4,3);
\draw[dotted] (0,2)--(4,2); \draw[dotted] (0,1)--(4,1);
\draw[dotted] (2,0)--(2,3); \draw[thick] (0,2)--(4,0);
\draw[thick] (2,2)--(4,3); \draw[thick] (0,0)--(2,1);
\end{tikzpicture}
\hspace{10mm}
\begin{tikzpicture}[scale=0.8]
\draw (0,0) rectangle (4,3);
\draw[dotted] (0,2)--(4,2); \draw[dotted] (0,1)--(4,1);
\draw[dotted] (2,0)--(2,3); \draw[thick] (0,2)--(4,0);
\draw[thick] (2,2)--(4,3); \draw[thick] (0,0)--(2,1);
\def\xa{0.25}
\def\xb{0.75}
\def\xc{1.25}
\def\xd{1.75}
\def\xe{2.25}
\def\xf{2.75}
\def\xg{3.25}
\def\xh{3.75}
\def\ya{2-\xa/2}
\def\yd{2-\xd/2}
\def\yg{1+\xg/2}
\def\yb{\xb/2}
\def\yc{\xc/2}
\def\ye{2-\xe/2}
\def\yf{1+\xf/2}
\def\yh{2-\xh/2}
\draw[fill,blue] (\xa,\ya) circle (0.1); \draw[fill,blue]
(\xb,\yb) circle (0.1); \draw[fill,blue] (\xc,\yc) circle (0.1);
\draw[fill,blue] (\xd,\yd) circle (0.1); \draw[fill,blue]
(\xe,\ye) circle (0.1); \draw[fill,blue] (\xf,\yf) circle (0.1);
\draw[fill,blue] (\xg,\yg) circle (0.1);
\draw[fill,blue] (\xh,\yh) circle (0.1);
\draw[blue] (\xa,\ya) node[above] {6}; \draw[ blue]
(\xb,\yb) node[above] {2}; \draw[ blue] (\xc,\yc) node[above] {3};
\draw[ blue] (\xd,\yd) node[above] {5}; \draw[ blue]
(\xe,\ye) node[above] {4}; \draw[ blue] (\xf,\yf) node[above] {7};
\draw[ blue] (\xg,\yg) node[above] {8};
\draw[ blue] (\xh,\yh) node[above] {1};
\end{tikzpicture}
\caption{A matrix $M$, its corresponding grid of segments, and a drawing of the permutation $62354781\in\G_8(M)$.}
\label{fig:grid}
\end{figure}
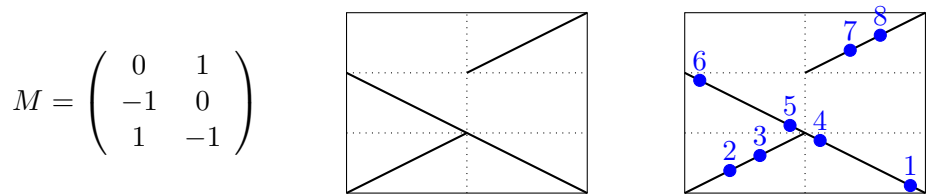

\subsection{Summary of main results}

In general, the product of Schur-positive subsets of $\S_n$ does
not give a Schur-positive multiset or set. We are interested in
finding families of subsets whose product
is Schur-positive, either as a multiset or as a set.

After introducing some definitions in Section~\ref{sec:prelim}, a little background on fine sets in Section~\ref{sec:background},
and some simple examples of Schur-positive grid classes in Section~\ref{sec:elementary},
our first main result about products of Schur-positive sets appears in Section~\ref{sec:products} as Theorem~\ref{main}. For each $J\subseteq \{1,\dots,n-1\}$, define the {\em descent
class} $D_{n,J}:=\{\pi\in \S_n:\ \Des(\pi)=J\}$, and its inverse $D_{n,J}^{-1}:=\{\pi^{-1}:\ \pi\in D_{n,J}\}$.
It was shown by Gessel~\cite{Gessel} that $D_{n,J}^{-1}$ is Schur-positive.
Theorem~\ref{main} includes the following statement, which appears also as Theorem~\ref{main_FD}.

\begin{theorem}\label{main_FD_1}
For every Schur-positive set $\BBB\subseteq \S_n$ and every
$J\subseteq \{1,\dots,n-1\}$, the multiset product $\BBB
D_{n,J}^{-1}$ is Schur-positive.
\end{theorem}

We provide a representation-theoretic proof that involves
Solomon's descent representations
and Stanley's shuffling theorem.

If instead of considering multiset products we are interested in the underlying sets being Schur-positive, we have a more restricted theorem.
Let $c$ be the $n$-cycle $(1,2,\dots,n)$, and let $C_n=\langle
c\rangle=\{c^k:\ 0\le k< n\}$, the cyclic subgroup generated by $c$, which is shown to be Schur-positive in Corollary~\ref{cor:LLLfine}. We can interpret
$D_{n-1,J}^{-1}$ as a subset of $\S_n$ by identifying $\S_{n-1}$ as the set of the permutations in $\S_n$ that fix $n$. The following theorem about multiset products appears as Theorem~\ref{thm:horizontal1}.

\begin{theorem}\label{main_FD_2}
For every $J\subseteq \{1,\dots,n-2\}$, the set product
$D_{n-1,J}^{-1}C_n$ is Schur-positive.
\end{theorem}

The proof of this theorem combines a descent-set-preserving bijection together with sieve methods and
representation-theoretic arguments.

Schur-positivity of various set and multiset products of grid classes follows from the above two theorems.
In Section~\ref{sec:vertical} we discuss applications of Theorem~\ref{main_FD_1} to vertical rotations of grids, of which arc permutations are a special case, and thus we obtain a short proof of their Schur-positivity. Applications of Theorem~\ref{main_FD_2} to horizontal rotations are discussed in Section~\ref{sec:horizontal}, together with
equidistribution phenomena.

Theorems~\ref{main_FD_1} and~\ref{main_FD_2}  imply
Schur-positivity of certain  grid classes. One such class is $\G(M_k)$, the grid obtained by vertical rotation of the
grid whose matrix is a $k\times 1$ matrix of ones (see the drawing of the grid $\G(M_3)$ in Figure~\ref{fig:GM3}).
The following consequence is stated in Corollary~\ref{cor:rotated_shuffles2}.

\begin{corollary}
For every positive integers $k$ and $n$, $\Q(\G_n(M_k))$ is
Schur-positive.
\end{corollary}

A closely related application involves {\em cyclic descents},
which were introduced by Cellini~\cite{Cellini} and studied by
Petersen, Dilks and Stembridge~\cite{Petersen, Dilks} (see
Definition~\ref{def_cyc_des}).  Even though inverse cyclic descent
classes are not necessarily Schur-positive,
it follows from Theorem~\ref{main_FD_2} that subsets of
permutations with fixed inverse cyclic descent number are
Schur-positive (Corollary~\ref{cor:cyc_fine}).

Other geometric operations on grid classes, such as reflections
and stacking, are applied in Section~\ref{sec:other} to get
more examples of  Schur-positive grid classes.
Section~\ref{prod_one_col} presents an explicit compact description of the set product of a one-column
grid class with an arbitrary grid class. This result generalizes~\cite[Theorem 6]{AAA}.
The paper concludes with a list of some open problems, questions, conjectures, and ideas for further work in Section~\ref{sec:open}.

Table~\ref{tab:summary} summarizes our main results and conjectures about Schur-positive sets and multisets, together with their location in the paper.

\begin{table}[h]
\centering
\begin{tabular}{|l|l|}
\hline
\multicolumn{2}{|c|}{\bf Known Schur-positive sets} \\
\hline
Inverse descent classes $D_{n,J}^{-1}$ & Lemma~\ref{lem:DJ} \cite[Thm. 7]{Gessel}\\
\hline
Subsets of fixed inversion number & \multirow{3}{*}{Proposition \ref{fine_list} \cite[Prop. 9.5]{Adin-R}}\\
\cline{1-1}
Subsets closed under conjugation & \\
\cline{1-1}
Subsets closed under Knuth relations & \\
\hline
Arc permutations $\A_n$ & Proposition~\ref{prop:rotated_leftunimodal} \cite[Thm. 5]{ER1} \\
\hline\hline
\multicolumn{2}{|c|}{\bf New Schur-positive sets} \\
\hline
One-column grid classes $\G_n^{\vv}$ (includes left-unimodal  & \multirow{2}{*}{Corollary \ref{one-column_zigzags}}\\
permutations $\LL_n$ and shuffles of increasing sequences $\G_n^{+^k}$) & \\
\hline
Colayered permutations $\LLL_n^k$ & Corollary \ref{cor:LLLfine}\\
\hline
Shuffles of Schur-positive sets on disjoint symmetric groups & Lemma \ref{lemma_shuffles}\\
\hline
Vertical rotations of grids of the form $\{C_n \G_n^{+^k}\}=\G_n(M_k)$ & Proposition~\ref{R2}, Corollary~\ref{cor:rotated_shuffles2}\\
\hline
Horizontal rotations $D_{n-1,J}^{-1}C_n$ of inverse descent classes in $\S_{n-1}$& Theorem~\ref{thm:horizontal1}\\
\hline
Sets with a fixed number of inverse cyclic descents & Corollary~\ref{cor:cyc_fine}\\
\hline
Certain stacked grids $\mathcal{J}_n$ and $\mathcal{K}_n$ & Propositions~\ref{J:prop} and~\ref{K:prop}\\
\hline
* Vertical rotations of one-column grids  $\{C_n \G_n^{\vv}\}$ & Conjecture~\ref{conj:rotations-onecolumn}\\
\hline
* Vertical rotations $C_nD_{n-1,J}^{-1}$ of inverse descent classes in $\S_{n-1}$ & Conjecture~\ref{Conj_CD_n-1}\\
\hline
Horizontal rotations $\BBB C_n$ of Schur-positive sets $\BBB\subseteq\S_{n-1}$& Theorem~\ref{thm:horizontal_induction}\\
\hline\hline
\multicolumn{2}{|c|}{\bf New Schur-positive multisets} \\
\hline
Multiset products $\BBB D_{n,J}^{-1}$ for every Schur-positive set $\BBB\subseteq\S_n$ & Theorem~\ref{main}\\
\hline
Vertical rotations  $C_n D_{n,J}^{-1}$ of inverse decent classes & Corollary~\ref{cor:vertical}\\
\hline
Multiset products  $\LLL^k_n \G_n^\vv$ of colayered perm. with one-column grids & Corollary~\ref{cor:LLLGv}\\
\hline
* Multiset products $D_{n,J}^{-1}\BBB$ for every Schur-positive set $\BBB\subseteq\S_n$ & Conjecture~\ref{conj:desfine}\\
\hline
\end{tabular}
\caption{Summary of results about Schur-positive sets and multisets. Conjectured statements are marked with a *.}
\label{tab:summary}
\end{table}

\section{Preliminaries}\label{sec:prelim}

\subsection{Descent classes and ribbons}\label{sec:zigzag}

Let $[n]:=\{1,2,\dots,n\}$, and let $\S_n$ denote the symmetric group on $[n]$.
The {\em descent set} of a permutation $\pi\in \S_n$ is defined by
\[
\Des(\pi):=\{i:\ \pi(i)>\pi(i+1)\}.
\]
For $J\subseteq [n-1]$, define the {\em descent class}
\[
D_{n,J}:=\{\pi\in \S_n:\ \Des(\pi)=J\}
\]
and the corresponding {\em inverse descent class} $D_{n,J}^{-1}:=\{\pi^{-1}:\pi\in D_{n,J}\}$.

For a skew shape $\lambda/\mu$, let $\SYT(\lambda/\mu)$ be the set
of standard Young tableaux of shape $\lambda/\mu$. We use the
English notation in which row indices increase from top to bottom.

For $T\in\SYT(\lambda/\mu)$, define its descent set by
\[
\Des(T):= \{ i:\ i+1 \textrm{ lies southwest of } i \textrm { in $T$}\}.
\]

A {\em ribbon} or {\em zigzag shape} is a connected skew shape that does not
contain a $2\times 2$ square.
 For example, every hook is a ribbon.
There is a natural bijection between the set of all subsets of
$[n-1]$ and the set of all ribbons of size $n$, where the size is defined to be the number of cells.

\begin{defn}\label{zigzag_subset}
Given a subset $J \subseteq [n-1]$, let $Z_{n,J}$ be the ribbon with $n$ cells labeled $1,\dots,n$ increasing in the northeast direction, where cell $i+1$ is immediately above cell $i$ if $i \in J$, and immediately to the right of cell $i$ otherwise.
\end{defn}

For example, if $n = 9$ and $J = \{1,3,5,6\}$, then $$Z_{n,J} =
\young(::\hfil\hfil\hfil,::\hfil,:\hfil\hfil,\hfil\hfil,\hfil).$$

Consider the map from the set of standard Young tableaux (SYT for short) of all ribbons of size $n$ to
permutations in $\S_n$ defined by taking the reading word of the SYT, that is, listing its entries starting
from the southwest corner and moving along the shape. The restriction of
this map to the set $\SYT(Z_{n,J})$ of tableaux of a fixed ribbon shape $Z_{n,J}$ is a bijection
to permutations in $\S_n$ with descent set $J$.
This bijection has the property that the descent set of the SYT becomes the descent set of the inverse of the associated permutation.
To see this, suppose that the entries of a SYT $T$ starting from the SW corner are $\pi=\pi(1)\pi(2)\dots\pi(n)$. Then, $i$ is a descent of $\pi^{-1}$ if $i+1$ appears to the left of $i$ in this sequence, which is equivalent to saying that $i+1$ appears in a box lower than $i$ in $T$, that is, $i$ is a descent of $T$.
We conclude that for every $J\subseteq [n-1]$, the distribution of the statistic $\Des$ is the same over $D_{n,J}^{-1}$ and over $\SYT(Z_{n,J})$.

\subsection{Knuth classes and shuffles}\label{prel_sec:shuffles}

The well-known Robinson--Schensted correspondence maps each
permutation $\pi \in \S_n$ to a pair $(P_\pi, Q_\pi)$ of standard
Young tableaux of the same shape $\lambda \vdash n$.
Recall that this correspondence is a $\Des$-preserving bijection in the following sense.

\begin{lemma}[{\cite[Lemma 7.23.1]{Stanley_ECII}}]\label{Knuth1}
For every permutation $\pi\in \S_n$,
\[
\Des(P_\pi)=\Des(\pi^{-1}) \quad \text{ and } \quad
\Des(Q_\pi)=\Des(\pi).
\]
\end{lemma}

The {\em Knuth class} corresponding to a standard Young tableau
$T$ of size $n$ is the set of permutations
\[
{\mathcal C}_T=\{ \pi \in \S_n :\ P_\pi = T \}.
\]
If $T$ has shape $\lambda$ then we say that ${\mathcal C}_T$ is a {\em Knuth
class of shape $\lambda$}.
By Lemma~\ref{Knuth1}, inverse descent classes are disjoint unions
of Knuth classes.

A {\em Knuth relation} on a permutation $\pi$ is a switch of two adjacent letters $ac$, where
$bac$ or $acb$ are adjacent in $\pi$ and $a<b<c$ or $c<b<a$. Two permutations in $\S_n$ are {\em Knuth-equivalent}
if one can be obtained from the other by applying a sequence of Knuth relations.

\begin{theorem}[{\cite[Theorem 3.4.3]{Sagan_book}}]
The Knuth-equivalence classes resulting from the above Knuth relations on $\S_n$ are precisely the Knuth classes.
\end{theorem}

\begin{defn}
Let $U$ and $V$ be disjoint sets of letters and let $\sigma$ and
$\tau$ be two permutations of $U$ and $V$ respectively. A {\em
shuffle of $\sigma$ and $\tau$}, denoted by $\sigma\shuffle \tau$,
is the set of all permutations of $U\sqcup V$ where the
letters of $U$ appear in same order as in $\sigma$ and the
letters of $V$ appear in same order as in $\tau$. For sets
(or multisets) $A$ and $B$ of permutations on disjoint finite sets
of letters, denote by $A\shuffle B$ the (multi)set of all shuffles
of a permutation in $A$ with a permutation in $B$.
\end{defn}

For example,  if
$A$ and $B$ are the multisets $A=\{\{12,12\}\}$ and $B=\{\{3\}\}$,
then $A\shuffle B=\{\{312,312,132,132,123,123\}\}$; if $A$ and $B$
are the sets $A=\{12,21\}$ and $B=\{43\}$, then $A\shuffle
B=\{1243,1423,1432,4123,4132,4312,
2143,2413,2431,4213,4231,4321\}$.
Shuffles will play an important role in Section~\ref{sec:products}.

For partitions $\mu\vdash k$ and $\nu\vdash n-k$, let  $(\mu,\nu)$
be the skew Young diagram obtained by placing Young diagrams of shape $\mu$ and $\nu$
so that the upper-right vertex of the Young diagram of shape $\mu$ coincides with the lower-left vertex of the Young diagram of shape $\nu$.
A {\em Knuth class on the letters $k+1,\dots,n$}
is an equivalence class of the symmetric group on these letters resulting from the Knuth relations
(equivalently, a set obtained from a Knuth class in $\S_{n-k}$ by shifting the letters up by~$k$).
The following result belongs to mathematical folklore. Some generalizations of it can be found in~\cite{BV,R_ind}.

\begin{theorem}\label{BV}
Let $A$ be a Knuth class of shape $\mu$ on the letters $1,\dots,
k$ and $B$ a Knuth class of shape $\nu$ on the letters
$k+1,\dots,n$. The following hold.
\begin{enumerate}
\item $A\shuffle B$ is closed under Knuth relations.
\item The distribution of $\Des$ is the same over $A\shuffle B$ and over $\SYT((\mu,\nu))$.
\end{enumerate}
\end{theorem}

\begin{proof}
For the first part, we will prove that for $\pi\in
A\shuffle B$ containing adjacent letters $bac$ with $1\le
a<b<c\le n$, switching $a$ and $c$ gives another permutation $\bar \pi\in
A\shuffle B$. The other cases are identical. First notice that there exist unique $\sigma\in A$
and $\tau\in B$ such that $\pi\in \sigma\shuffle \tau$. If $a<k<c$,
then clearly $\bar\pi\in \sigma\shuffle \tau$. If $c\le k$, then
 $bac$ is a sequence of adjacent letters in $\sigma$, so switching $a$ and $c$ in
 $\sigma$ we get a permutation $\bar\sigma\in A$ such that $\bar\pi\in
\bar\sigma \shuffle \tau$. The case $k<a$ is similar, completing the
proof of part 1.

For the second part, define a map $f:A\shuffle B\longrightarrow
\SYT((\mu,\nu))$ as follows. For a permutation $\pi\in A\shuffle B$,
there exist unique $\sigma\in A$ and $\tau\in B$ such that $\pi\in
\sigma\shuffle \tau$. Let $Q_\sigma$ and $Q_\tau$ be the recording tableaux
that correspond to $\sigma$ and $\tau$ under Robinson--Schensted,
which have shape $\mu$ and $\nu$, respectively.

Let $a_1 < \cdots <a_k$ be the list of positions of the letters of
$\sigma$ in $\pi$, that is, the ordered set $\{\pi^{-1}(i):\ 1\le
i\le k\}$. Similarly, let $b_1 < \cdots <b_{n-k}$ be the list of
positions of the letters of $\tau$ in $\pi$. Place $Q_\sigma$ and
$Q_\tau$ so that the upper-right vertex of $Q_\sigma$ and the lower-left vertex of
$Q_\tau$ coincide, and replace each letter $i$ in the resulting
skew SYT as follows: if $i\le k$ replace $i$ by $a_i$, if $k<i$
replace $i$ by $b_{i-k}$. For example, if $\pi=16783245$ and
$k=3$, we have $\sigma=132$, $\tau=67845$ and
\[
f(\pi)=\young(::234,::78,15,6).
\]
One can verify that $f$ is a $\Des$-preserving bijection.
\end{proof}

The following result, due to Stanley, will be used in Section~\ref{sec:products}. For bijective proofs, see~\cite{Goulden, Stadler}.

\begin{proposition}[{\cite[Ex. 3.161]{ECI}}]\label{Stanley_lemma}
Given two permutations $\sigma$ and $\tau$ of disjoint sets of
integers, the distribution of the descent set over all shuffles of
$\sigma$ and $\tau$ depends only on $\Des(\sigma)$ and $\Des(\tau)$.
\end{proposition}

\begin{remark}\label{m_shuffles}
Theorem~\ref{BV} and Proposition~\ref{Stanley_lemma} may be
generalized to shuffles of any number of Knuth classes and
permutations, respectively.
\end{remark}

\subsection{Symmetric and quasisymmetric functions}\label{sec:prel_quasi}

Schur functions indexed by partitions of $n$ form a
distinguished basis for $\Lambda^n$, the vector space of homogeneous
symmetric functions of degree $n$;
see, e.g., \cite[Corollary 7.10.6]{Stanley_ECII}.
A symmetric function is {\em Schur-positive} if all the coefficients in its expansion
in the basis of Schur functions are nonnegative.

The irreducible characters of $\S_n$ are also indexed by partitions of $n$.
The {\em Frobenius image} of an $\S_n$-character
$\chi=\sum\limits_{\lambda\vdash n} c_\lambda \chi^\lambda$ is the
symmetric function $f=\sum\limits_{\lambda\vdash n} c_\lambda
s_\lambda$, which we denote by $\ch(\chi)$.

It is clear from the combinatorial definition of Schur functions~\cite[Defintion 7.10.1
and discussion in p. 339]{Stanley_ECII}
that for every pair of partitions $\mu$ and $\nu$ we have
\begin{equation}\label{eq:smunu}
s_{(\mu,\nu)}=s_\mu s_\nu.
\end{equation}
Since the characteristic map is an algebra
isomorphism \cite[Proposition 7.18.2]{Stanley_ECII}, it follows
that if the size of the skew shape $(\mu,\nu)$ is $n$, then
$\chi^{(\mu,\nu)}=(\chi^\mu\otimes \chi^\nu)\uparrow^{\S_n}$. This
identity is naturally generalized to skew shapes with $t$
disconnected components as follows.

\begin{lemma}\label{eq_skew_induced}
Consider a skew shape $\lambda/\mu$ of size $n$, which consists of
disconnected diagrams $\nu^{(1)}\vdash k_1, \dots ,
\nu^{(t)}\vdash k_t$.
Then $s_{\lambda/\mu}= s_{\nu^{(1)}}s_{\nu^{(2)}}\cdots
s_{\nu^{(t)}}$ and
\[
\chi^{\lambda/\mu}= (\chi^{\nu^{(1)}}\otimes
\chi^{\nu^{(2)}}\otimes \cdots \otimes
\chi^{\nu^{(t)}})\uparrow_{\S_{k_1}\times \S_{k_2}\times \cdots
\times \S_{k_t}}^{\S_n}.
\]
\end{lemma}

Of special interest is the case of two connected components, 
the first consisting of one box and the second being an arbitrary
partition $\mu\vdash n-1$. Then $\chi^{((1),\mu)}$ is equal to the
induced character $\chi^{\mu}\uparrow^{\S_n}$, which amounts
to the multiplicity-free sum of all
irreducible characters indexed by partitions obtained by adding a box 
to the Young diagram of $\mu$~\cite[Theorem
2.8.3]{Sagan_book}. Thus, for every $\mu\vdash n-1$,
\begin{equation}\label{Pierri_1}
s_{1} s_\mu = \sum\limits_{\lambda \vdash n\atop \lambda/\mu=(1)}
s_\lambda.
\end{equation}

\medskip

{\em Quasisymmetric functions} were introduced by Gessel, see \cite[Section 7.19]{Stanley_ECII} for definitions and background.
For each subset $D \subseteq [n-1]$, define the quasisymmetric function
\[
F_{n,D} := \sum\limits_{i_1\le i_2 \le \ldots \le i_n \atop {i_j <
i_{j+1} \text{ if } j \in D}} x_{i_1} x_{i_2} \cdots x_{i_n}.
\]
The set $\{F_{n,D}:D \subseteq [n-1]\}$ is a basis of the vector space of homogeneous quasisymmetric functions of degree $n$.

Let $\BBB$ be a (multi)set of combinatorial objects, equipped with a
descent map $\Des$ which associates with each
element $b\in \BBB$ a subset $\Des(b) \subseteq [n-1]$. Define the
quasisymmetric function
\begin{equation}\label{eq:QB}
\Q(\BBB) := \sum\limits_{b\in \BBB} m(b,\BBB) F_{n,\Des(b)},
\end{equation}
where $m(b,\BBB)$ is the multiplicity of the element $b$ in
$\BBB$. Note that for two (multi)sets $\BBB$ and $\BBB'$, the equality $\Q(\BBB)=\Q(\BBB')$ is equivalent to the fact that $\Des$ has the same distribution over $\BBB$ and over~$\BBB'$.

\medskip

The following key observation is due to Gessel.

\begin{proposition}[{\cite[Theorem 7.19.7]{Stanley_ECII}}]\label{G1} For every skew shape $\lambda/\mu$,
\[
\Q({\SYT(\lambda/\mu)})=s_{\lambda/\mu}.
\]
\end{proposition}

The next result about inverse descent classes is again due to Gessel. Here we include a proof because the same idea will be used in Section~\ref{sec:one-col}. For a symmetric function $f=\sum_{\lambda\vdash n} c_\lambda
s_\lambda$, let $\langle f, s_\mu\rangle=c_\lambda$.

\begin{lemma}[{\cite[Theorem 7]{Gessel}}]\label{lem:DJ}
For $J\subseteq[n-1]$ and $\lambda\vdash n$,
\[
\Q(D_{n,J}^{-1})=\Q(\SYT(Z_{n,J}))=s_{Z_{n,J}} \] and
\[
\langle  \Q(D_{n,J}^{-1}), s_\lambda\rangle
=|\{P\in\SYT(\lambda):\Des(P)=J\}|.
\]
\end{lemma}

\begin{proof}
The expression as a skew Schur function follows from the $\Des$-preserving bijection described at the end of Section~\ref{sec:zigzag} 
and Proposition~\ref{G1}, which imply that
$\Q(D_{n,J}^{-1})=\Q(\SYT(Z_{n,J}))=s_{Z_{n,J}}$.

The expansion as a sum of Schur functions can be obtained by noting, using Lemma~\ref{Knuth1},
that the image of $D_{n,J}^{-1}$ by RSK is the set of pairs of
standard Young tableaux $(P,Q)$ with $\shape(P)=\shape(Q)\vdash n$
and $\Des(P)=J$. By Lemma~\ref{Knuth1} and Proposition~\ref{G1},
it follows that
\begin{align*}\Q(D_{n,J}^{-1})=\sum_{\pi\in D_{n,J}^{-1}} F_{n,\Des(\pi)}&= \sum_{\lambda\vdash n} \sum_{P\in\SYT(\lambda) \atop \Des(P)=J} \sum_{Q\in\SYT(\lambda)} F_{n,\Des(Q)}\\
&=\sum_{\lambda\vdash n} |\{P\in\SYT(\lambda):\Des(P)=J\}|\,
s_\lambda.\qedhere\end{align*}
\end{proof}

In particular, inverse descent classes are symmetric and Schur-positive. The following variation of Problem~\ref{prob:symmetric} was
proposed in~\cite{Adin-R}.

\begin{problem}\label{prob:schurpositive}
For which subsets of permutations $A \subseteq \S_n$ is $\Q(A)$
Schur-positive?
\end{problem}

\subsection{Grid classes}\label{sec:grid}

A useful tool in our construction of Schur-positive sets is the
concept of {\em geometric grid classes}, introduced and studied by
Albert et al.~\cite{AABRV}.
To each matrix $M$ with entries in $\{0,1,-1\}$, we associate the grid obtained
by placing line segments of slope $1$ and $-1$ in the locations of the ones and negative ones in $M$, respectively.
See Figure~\ref{fig:grid} for an example.
A geometric grid class consists of those permutations that can be drawn on such a grid.

\begin{defn}
\begin{enumerate}
\item For any matrix $M$ with entries in $\{0,1,-1\}$, let
 $\G_n(M)$ be the set of permutations in
$\S_n$ that can be obtained by placing $n$ dots on the segments of the grid corresponding to $M$ (in
such a way that no two dots have the same $x$- or $y$-coordinate),
labeling the dots with $1,2,\dots,n$ by
increasing $y$-coordinate, and then reading them by increasing
$x$-coordinate.

\item  Let $\G(M)=\bigcup_{n\ge0}\G_n(M)$. We call $\G(M)$ a {\em geometric grid class}, or simply a {\em grid class} for short. All grid classes that appear in this paper are geometric grid classes.
\end{enumerate}
\end{defn}

\begin{ex}
{\em Left-unimodal permutations}, defined as those for which every prefix forms an interval in $\bbz$, are
those in the grid class
$$\G\left(\ba{c} 1\\ -1 \ea\right).$$ A drawing of a permutation
on this grid is shown in Figure~\ref{fig:grid_staircase}. Denote by $\LL$ the set of left-unimodal permutations.
\end{ex}

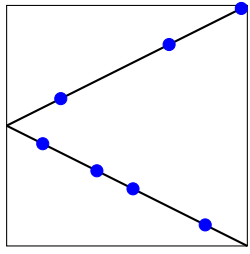
\begin{figure}[htb]
\centering
\begin{tikzpicture}[scale=0.8]
\draw (0,0) rectangle (4,4); \draw[thick] (0,2)--(4,4);
\draw[thick] (0,2)--(4,0);
\def\xa{6/10}
\def\xb{9/10}
\def\xc{15/10}
\def\xd{21/10}
\def\xe{27/10}
\def\xf{33/10}
\def\xg{39/10}
\def\ya{2-\xa/2}
\def\yc{2-\xc/2}
\def\yd{2-\xd/2}
\def\yf{2-\xf/2}
\def\yb{2+\xb/2}
\def\ye{2+\xe/2}
\def\yg{2+\xg/2}
\draw[fill,blue] (\xa,\ya) circle (0.1); \draw[fill,blue]
(\xb,\yb) circle (0.1); \draw[fill,blue] (\xc,\yc) circle (0.1);
\draw[fill,blue] (\xd,\yd) circle (0.1); \draw[fill,blue]
(\xe,\ye) circle (0.1); \draw[fill,blue] (\xf,\yf) circle (0.1);
\draw[fill,blue] (\xg,\yg) circle (0.1);
\end{tikzpicture}
\caption{A drawing of the permutation $4532617$ on the grid for left-unimodal
permutations.}
\label{fig:grid_staircase}
\end{figure}

Let $\st$ be the standardization operation that, given a sequence
of distinct integers, replaces the smallest entry with a $1$, the
second smallest with a $2$, and so on. A permutation $\pi$ {\em contains} another permutation $\sigma$ if there is a subsequence of $\pi$ whose standardization is $\sigma$. For example, $6425173$ contains $2314$ because $\st(4517)=2314$.

Since removing dots from the drawing of a permutation on a grid yields drawings of the permutations that it contains, it is clear that every geometric grid class is closed under pattern containment. Thus, it is characterized
by its set of minimal forbidden patterns, which is always finite, as shown in~\cite{AABRV}.

\section{Background on fine sets}\label{sec:background}

\subsection{Fine sets and Schur-positivity}

\begin{defn}
\begin{enumerate}
\item A set (or multiset)  of combinatorial objects $\BBB$
equipped with a descent map $\Des:\BBB\to 2^{[n-1]}$ is
called {\em set-fine} (or {\em multiset-fine})  if the
quasisymmetric function $\Q(\BBB)$ in Equation~\eqref{eq:QB}
is symmetric and Schur-positive. We use the term {\em fine} to mean set-fine or multiset-fine when there is no confusion.

\item[2.] The (multi)set $\BBB$ is {\em fine for a
complex $\S_n$-representation $\rho$} if
\[
\ch(\chi^\rho)=\Q(\BBB).
\]
\end{enumerate}
\end{defn}

This paper considers fine sets and multisets of permutations and standard Young
tableaux.
Since the sum of Schur-positive symmetric functions is
Schur-positive, we have the following.

\begin{observation}\label{union}
Unions of disjoint fine sets, as well as unions of fine multisets, are fine.
\end{observation}

We say that a sequence $(a_1, \ldots, a_n)$ of distinct positive
integers is {\em co-unimodal} if there exists $1 \le m\le n$ such
that
\[
a_1 > a_2 > \ldots > a_m < a_{m+1} < \ldots < a_n.
\]
Let $\mu=(\mu_1,\dots,\mu_t)$ be a composition of $n$. A sequence
of $n$ positive integers is called {\em $\mu$-modal} if the first
$\mu_1$ entries form a co-unimodal sequence, the next $\mu_2$
entries form a co-unimodal sequence, and so on. A permutation $\pi
\in \S_n$ is called {\em $\mu$-modal} if the sequence $(\pi(1),
\ldots, \pi(n))$ is $\mu$-modal. An element $b\in \BBB$ is {\em
$\mu$-modal} if there exists a $\mu$-modal permutation $\pi\in
\S_n$, such that $\Des(b)=\Des(\pi)$. Denote by $\BBB_\mu$ the sub(multi)set
of all $\mu$-modal elements in $\BBB$.

\begin{ex}
The standard Young tableau
\[
T=\young(1358,247,6)
\]
is $(3,1,4)$-modal, since $\pi=21346578\in \S_8$ is
$(3,1,4)$-modal and $\Des(T)=\{1,3,5\}=\Des(\pi)$.
\end{ex}

Let $\Comp(n)$ denote the set of all compositions of $n$. For
$\mu\in\Comp(n)$, let $S(\mu) := \{\mu_1, \mu_1+\mu_2, \ldots,
\mu_1+\dots+\mu_{t-1}\}$.

\begin{theorem}[{\cite[Theorem 1.5]{Adin-R}}]\label{m1}
For every set (or multiset)
of combinatorial objects $\BBB$ equipped with a descent
map $\Des:\BBB\to 2^{[n-1]}$, the following are equivalent.
\begin{enumerate}[(i)]
\item $\BBB$ is fine, that is, the quasisymmetric function
$\Q(\BBB)$ is symmetric and Schur-positive.

\item The function $\chi^\BBB : \Comp(n) \longrightarrow \ZZ$
defined by
\[
\chi^\BBB(\mu):=\sum\limits_{b\in \BBB_\mu} m(b,\BBB)
(-1)^{|\Des(b)\setminus S(\mu)|}
\]
is an $\S_n$-character, i.e., it does not depend on the order of
parts in $\mu$ and is a linear combination, with nonnegative
integer coefficients, of the irreducible characters of $\S_n$.

\item There exist a (multi)set partition $\BBB = \BBB_1 \sqcup
\ldots \sqcup \BBB_m$ and $\Des$-preserving bijections from each
$\BBB_i$ to the set of all standard Young tableaux of some shape
$\lambda^{(i)} \vdash n$, for each $1 \le i \le m$.

\item There exists a basis
$\{C_b : b\in \BBB\}$ for an $\S_n$-representation space $V$ such
that the linear action of every simple reflection $s_i$ has the
form
\[
s_i (C_b) =
\begin{cases}
-C_b, &\text{if } i \in \Des (b); \\
C_b + \displaystyle\sum_{b' \in \BBB \,:\, i \in\Des(b')}
a_i(b,b') C_{b'}, &\text{otherwise,}
\end{cases}
\]
for suitable coefficients $a_i(b,b')$.
\end{enumerate}
\end{theorem}

\begin{remark}\label{clarification}
By Lemma~\ref{Knuth1}, the distribution of the descent set over
all standard Young tableaux of some shape $\lambda$ is equal to
its distribution over all permutations in a Knuth class of shape
$\lambda$. Thus, condition (iii) in Theorem~\ref{m1} is
equivalent to the existence of a (multi)set partition $\BBB =
\BBB_1 \sqcup \ldots \sqcup \BBB_m$ and $\Des$-preserving
bijections from each $\BBB_i$ to the set of permutations in a
Knuth class of shape $\lambda^{(i)} \vdash n$, for each $1 \le i
\le m$.
\end{remark}

\begin{theorem}[{\cite[Theorem 3.2]{CA}}]\label{m2}
A set (or multiset) $\BBB$ is fine for an
$\S_n$-representation $\rho$, namely $\ch(\chi^\rho)=\Q(\BBB)$, if
and only if
\[
\chi^\rho(\mu)=\sum\limits_{b\in \BBB_\mu}
(-1)^{|\Des(b)\setminus S(\mu)|}
\]
for every partition $\mu\vdash n$.
\end{theorem}

\begin{remark}
Theorems~\ref{m1} and~\ref{m2} were proved in~\cite{Adin-R}
and~\cite{CA} for sets. It can be checked that the proofs
hold for multisets as well.
\end{remark}

The next proposition lists a few examples of set-fine subsets of
$\S_n$.

\begin{proposition}[{\cite[Prop. 9.5]{Adin-R}}]\label{fine_list}
The following subsets of $\S_n$ are fine:
\begin{enumerate}[(i)]
\item all permutations of fixed Coxeter length (equivalently,
fixed inversion number);
\item subsets closed under
conjugation (e.g., conjugacy classes, the set of involutions, and the set of permutations
with fixed cycle number);
\item subsets closed under Knuth
relations (e.g., Knuth classes, inverse descent classes, and the set of
permutations with fixed inverse descent number).
\end{enumerate}
\end{proposition}

\subsection{Arc permutations}\label{sec:arc}

An example of a fine set that does not fall in any of the cases in
Proposition~\ref{fine_list} is the set of arc permutations in
$\S_n$. In the following definition, an interval in $\bbz_n$ refers to the set obtained by reducing a interval in $\bbz$ modulo $n$.

\begin{defn}
A permutation $\pi\in \S_n$ is an {\em arc permutation}
if, for every $1\le j\le n$, the first $j$ letters in $\pi$ form
an interval in $\bbz_n$. Denote by $\A_n$ the set of arc permutations in $\S_n$.
\end{defn}

For example, the permutation $12543$ is an arc permutation in
$\S_5$, but $125436$ is not an arc permutation in $\S_6$, since
$\{1,2,5\}$ is an interval in $\bbz_5$ but not in $\bbz_6$.

Arc permutations were introduced in the study of flip graphs of polygon triangulations~\cite{TFT2}.
Some combinatorial properties of these permutations, including their description as a union of grid classes
and their descent set distribution, are studied in~\cite{ER1}. In particular, it follows from~\cite[Theorem 5]{ER1}
that $\A_n$ is a Schur-positive set.  One of
the goals of this paper is to explain this result by providing a
general recipe for constructing Schur-positive subsets of $\S_n$.

\section{Elementary examples of Schur-positive grid classes}\label{sec:elementary}

We remark that, by definition, grid classes are always sets of permutations, that is, there are no multiplicities.

\subsection{One-column grid classes}\label{sec:one-col}

Grid classes whose matrix consists of one column are particularly
interesting because they are unions of inverse descent
classes, as we will see in Proposition~\ref{one-column_descent_classes},
and thus fine sets by Lemma~\ref{lem:DJ}. Since we can ignore zero entries in the matrix without
changing the grid class, we will assume that the entries of the
matrix are in $\{1,-1\}$. If $\vv\in\{1,-1\}^k$, let $\G^\vv$
denote the one-column grid class where the entries of the matrix
are given by $\vv$ from bottom to top, and let $\G^\vv_n=\G^\vv\cap\S_n$. Sometimes we only write the
signs of the entries of $\vv$ for convenience. For example,
$\G^{-+}=\LL$ is the class of left-unimodal permutations, and
$\G^{++}$ is the class of shuffles of two increasing sequences, with all the letters in one sequence being smaller than all the letters in the other. In
general, we denote by $\G^{+^k}$ the class of shuffles of $k$
increasing sequences, each sequence consisting of a consecutive set of letters.

The following observation will be useful throughout this section.

\begin{observation}\label{obs:inverse}
For a given grid class $\mathcal{H}$, the set $\{\pi^{-1}:\pi\in\mathcal{H}\}$ is a grid class, whose grid is obtained by reflecting the grid for $\mathcal{H}$ across the southwest-northeast diagonal.
\end{observation}

We can easily characterize $\G^{+^k}$ and $\G^{-^k}$ as unions of
inverse descent classes, allowing us to express $\Q(\G^{+^k})$ and
$\Q(\G^{-^k})$ in terms of skew Schur functions, as follows.

\begin{observation}\label{obs:Gk}
For every $\pi\in \S_n$ and $k\ge1$,
\[
\pi\in \G^{+^k}\ \Longleftrightarrow\ \des(\pi^{-1})\le k-1
\]
and
\[
\pi\in \G^{-^k}\ \Longleftrightarrow\ \des(\pi^{-1})\ge n-k.
\]
\end{observation}

In the following proposition, the height and the width of a ribbon refer to the number of rows and columns, respectively.

\begin{proposition}\label{prop:QGk}
$\Q(\G^{+^k})$ (resp. $\Q(\G^{-^k})$) is the multiplicity-free sum of symmetric
functions of ribbons of height (resp. width) at most $k$. In terms of Schur functions,
\begin{align*}\Q(\G^{+^k}_n)&=\sum_{\lambda\vdash n} |\{P\in\SYT(\lambda):|\Des(P)|\le k-1\}|\, s_\lambda,\\
\Q(\G^{-^k}_n)&=\sum_{\lambda\vdash n}
|\{P\in\SYT(\lambda):|\Des(P)|\ge n-k\}|\, s_\lambda.
\end{align*}
\end{proposition}

\begin{proof}
By Observation~\ref{obs:Gk},
$$\G^{+^k}_n=\bigsqcup_{J:|J|<k} D_{n,J}^{-1}.$$
Thus, by Lemma~\ref{lem:DJ},
$$\Q(\G^{+^k}_n)=\sum_{J:|J|<k} \Q(D_{n,J}^{-1})=\sum_{J:|J|<k} \Q(SYT(Z_{n,J}))=\sum_{J:|J|<k} s_{Z_{n,J}}$$
which proves the first statement for $\Q(\G^{+^k})$. The statement for $\Q(\G^{-^k})$ is proved similarly by taking the sum over sets $J$ with $|J|\ge n-k$. The expressions in terms of  Schur functions also follow from Lemma~\ref{lem:DJ}.
\end{proof}

We obtain a particularly simple expression when $k=2$.

\begin{corollary}\label{cor:QSh} We have
$$\Q(\Sh_n)=s_n+\sum_{a=1}^{\lfloor\frac{n}{2}\rfloor} (n-2a+1)s_{n-a,a}.$$
\end{corollary}

\begin{proof}
To apply Proposition~\ref{prop:QGk}, we find all SYT with at most one descent. The one-row tableau is the only one with no descents, and for each
$1\le a\le \lfloor\frac{n}{2}\rfloor$, there are $n-2a+1$ tableaux
$P$ with $\shape(P)=(n-a,a)$ and one descent.
\end{proof}

Recall the class of left-unimodal permutations from
Section~\ref{sec:grid}, and let $\LL_n=\LL\cap\S_n$.

\begin{lemma}\label{lem:QLL} We have
$$\Q(\LL_n)=\sum_{k=0}^{n-1} s_{n-k,1^k}.$$
\end{lemma}

\begin{proof}
By Observation~\ref{obs:inverse}, $\LL_n=\G^{-+}_n$ is a disjoint
union of inverse descent classes. Specifically,
\begin{equation}\label{Eq:LL}
\LL_n=\bigsqcup_{k=0}^{n-1} D_{n,[k]}^{-1}.
\end{equation}
 We apply
Lemma~\ref{lem:DJ}, noting that the only tableau $P$ with
$\Des(P)=[k]$ is the hook with entries $1,2,\dots,k,k+1$ in its
first column.
\end{proof}

General one-column grid classes can be expressed as unions of inverse descent classes as follows.

\begin{proposition}\label{one-column_descent_classes}
For every $\vv\in \{+,-\}^k$,
$$
\G^\vv_n=\S_n\setminus \bigsqcup_{\uu\in \{+,-\}^{n-1}\atop \vv\le \uu} D_{n,J_\uu}^{-1} =
\bigsqcup_{\uu\in \{+,-\}^{n-1}\atop \vv\nleq \uu} D_{n,J_\uu}^{-1},
$$
where $\vv\le\uu$ denotes that $\vv$ is a subsequence of $\uu$, and $J_\uu:=\{i:\
u_i=+\}$. Consequently, by Lemma~\ref{lem:DJ}, $\G^\vv_n$ is a fine set.
\end{proposition}

\begin{ex} We have $$\G^{-++}_5=\S_5\setminus\left(D_{5,\{2,3\}}^{-1} \sqcup D_{5,\{2,4\}}^{-1}\sqcup D_{5,\{3,4\}}^{-1} \sqcup D_{5,\{1,3,4\}}^{-1} \sqcup D_{5,\{2,3,4\}}^{-1}\right).$$
\end{ex}

\begin{proof}[Proof of Proposition~\ref{one-column_descent_classes}]
It will be convenient to record the descent set of a permutation in $\S_n$ as a word in $\{a,d\}^{n-1}$ which has $d$s exactly at the positions of the descents.
Let $W(\vv)$ be the set of words in $\{a,d\}^{n-1}$ that can be obtained from $\vv$ by replacing each maximal block of the form $+^r$ with a word with at most $r-1$ $d$s,
and each maximal block of the form $-^r$ with a word with at most $r-1$ $a$s. For example, $W(--+++-)$ is the set of words that can be obtained by concatenating a word with at most one $a$,
a word with at most two $d$s, and a word with no $a$s.

It follows easily from the definitions and Observation~\ref{obs:Gk} that $\pi\in\G^\vv_n$ if
and only if the descent word of $\pi^{-1}$ belongs to $W(\vv)$. On
the other hand, words in $W(\vv)$ are precisely those that avoid
the subsequence obtained from $\vv$ by replacing each $+$ with a
$d$ and each $-$ with an $a$. Indeed, this subsequence and any word containing it
cannot be decomposed as a concatenation of subwords as described above. Conversely, any word that avoids this subsequence can be broken up into subwords as follows: if $\vv$ starts with $+^r-^s\cdots$ (the other case if analogous), let the first break up point occur right before the $r$th occurrence of $d$ in the word (or at the end of the word if there are not enough occurrences), the second break up point right before the $s$th occurrence of $a$ after that, and so on.

It follows that $\pi\notin\G^\vv_n$ if and
only the descent word of $\pi^{-1}$ contains this subsequence.
This latter condition is equivalent to the fact that $\pi\in D_{n,J_\uu}^{-1}$ for
some $\uu$ that contains $\vv$ as a subsequence.
\end{proof}

Applying Lemma~\ref{lem:DJ} gives the following corollary.

\begin{corollary}\label{one-column_zigzags}
For every $\vv\in \{+,-\}^k$,
$$\Q(\G^\vv_n)=\sum_{\uu\in \{+,-\}^{n-1}\atop \vv\nleq \uu}s_{Z_{n,J_u}}.$$
\end{corollary}

For example, $\Q(\G^{-+})$ equals the multiplicity-free sum of ribbons
$Z_{n,J}$ where $J$ is a prefix of $1,\dots,n-1$. These ribbons are
precisely the hooks, agreeing with Lemma~\ref{lem:QLL}.

\medskip

Recall from~\cite[Ch. 9.4]{Reutenauer}
{\em the Solomon descent subalgebra} of $\CC[\S_n]$, which is
spanned by the elements of the group algebra
\begin{equation}\label{eq:dnJ}
d_{n,J}:=\sum\limits_{\pi \in D_{n,J}}\pi.
\end{equation}
For $\vv\in \{+,-\}^k$ and
$n>k$, let
\[
g_{n,\vv}:=\sum\limits_{\pi\in \G^\vv_n}\pi^{-1}\,\in\CC[\S_n].
\]
Letting $k=n-1$, we get the following immediate consequence of
Proposition~\ref{one-column_descent_classes}.

\begin{corollary}\label{Solomon_basis}
The set
\[
\{g_{n,\vv}:\ \vv\in\{+,-\}^{n-1}\}
\]
forms a basis for the Solomon descent subalgebra in $\CC[\S_n]$.
\end{corollary}

\begin{proof}
Proposition~\ref{one-column_descent_classes}, for the special case of
$\vv\in\{+,-\}^{n-1}$, states that $\G_{n}^\vv=\S_n\setminus
D_{n,J_\vv}^{-1}$. Since the set $\{d_{n,J}:\ J\subseteq
[n-1]\}$ forms a basis for the Solomon descent subalgebra, the result follows.
\end{proof}

\subsection{Colayered permutations}

Another family of Schur-positive grid classes is the following.

\begin{defn}
The {\em $k$-colayered grid} class $\LLL^k$ is determined by the $k\times
k$ identity matrix:
$$\LLL^k = \G\left(\begin{array}{ccccc}
           1 & 0 & 0 & \cdots & 0\\
           0 & 1 & 0 & \cdots & 0\\
           0 & 0 & 1 & \cdots & 0\\
           \vdots & \vdots & \vdots & \ddots & \vdots\\
           0 & 0 & 0 & \cdots & 1
         \end{array}\right).$$
Let $\LLL=\bigcup_{k\ge 1}\LLL^k$ be the set of colayered permutations with an arbitrary number of layers. Let $\LLL^k_n=\LLL^k\cap\S_n$ and
$\LLL_n=\LLL\cap\S_n$.
\end{defn}

Figure~\ref{fig:Y} shows the $2$- and $3$-colayered grids.

\begin{figure}[htb]
\centering
\begin{tikzpicture}[scale=.8]
\draw (0,0) rectangle (4,4); \draw[dotted] (0,2)--(4,2);
\draw[dotted] (2,0)--(2,4);
\draw[thick] (0,2)--(2,4); \draw[thick] (2,0)--(4,2);
\end{tikzpicture}
\hspace{10mm}
\begin{tikzpicture}[scale=.8]
\draw (0,0) rectangle (4,4); \draw[dotted] (0,1.333)--(4,1.333);
\draw[dotted] (0,2.666)--(4,2.666); \draw[dotted]
(1.333,0)--(1.333,4); \draw[dotted] (2.666,0)--(2.666,4);
\draw[thick] (0,2.666)--(1.333,4); \draw[thick]
(1.333,1.333)--(2.666,2.666); \draw[thick] (2.666,0)--(4,1.333);
\end{tikzpicture}
\caption{The grids for $\LLL^2$ (left) and $\LLL^3$ (right).}
\label{fig:Y}
\end{figure}

\begin{ex} We have that $\LLL^1_5=\{12345\}$ and
$\LLL^2_5=\{12345,51234,45123,34512,23451\}$.
\end{ex}

\begin{proposition}\label{prop:k-colayered_hooks}
For every $1\le k\le n$,
\[
\Q(\LLL^k_n\setminus \LLL^{k-1}_n)=s_{n-k+1,1^{k-1}},
\]
with the convention that $\LLL^{0}_n=\emptyset$.
\end{proposition}

\begin{proof}
A permutation $\pi\in \LLL^k$ does not belong to
$\LLL^{k-1}$ if and only if $\des(\pi)=k-1$.
Permutations $\pi\in \LLL^k$ with $\des(\pi)=k-1$ are in bijection with subsets of $[n-1]$ of cardinality $k-1$, via the map $\pi\mapsto\Des(\pi)$.
We conclude that
\[
\Q(\LLL^k_n\setminus \LLL^{k-1}_n)=\sum\limits_{\pi\in
\LLL^k_n\setminus \LLL^{k-1}_n}F_{n,\Des(\pi)}
=\sum\limits_{\substack{J\subseteq [n-1]\\ |J|=k-1}} F_{n,J}=
\sum\limits_{T\in
\SYT((n-k+1,1^{k-1}))}F_{n,\Des(T)}=s_{n-k+1,1^{k-1}}.
\qedhere\]
\end{proof}

\begin{corollary}\label{cor:LLLfine}
For every $1\le k\le n$, $\LLL^k_n$ is a fine set and
$$\Q(\LLL^k_n)=s_{n}+s_{n-1,1}+s_{n-2,1,1}+\dots+s_{n-k+1,1^{k-1}}.$$
In particular, $C_n$ is fine and $\Q(C_n)=s_{n}+s_{n-1,1}$.
\end{corollary}

\begin{remark}
$\Q(\LLL_n)$ is the Frobenius image of the character of
the exterior algebra $\wedge V$, where $V$ is the $n$-dimensional
natural representation space of $\S_n$.
This is because the exterior algebra $\wedge V$ is isomorphic to the
multiplicity-free sum of all hooks \cite[Exer. 4.6]{Fulton_Harris_book}.
\end{remark}

\section{Products of fine sets}\label{sec:products}

Given two subsets $A,B\subseteq\S_n$, define their product $AB$ to be the
multiset of all permutations obtained as $\pi\sigma$ where $\pi\in A$ and $\sigma\in B$. To denote the
underlying set, without multiplicities, we write $\{AB\}$. We call
$AB$ the multiset product and $\{AB\}$ the set product of $A$ and
$B$.

We are interested in pairs of fine subsets $A,B\subseteq \S_n$
whose product is fine, either as a multiset or as a set.

\subsection{Basic examples}

In general, the product of fine subsets of $\S_n$ does not give a
fine multiset or set. For example, the subsets
$A=\{2134,3412,1243\}$ and $B=\{2143,3412\}$ in $\S_4$ are fine,
but $AB=\{AB\}$ is not fine (although $\Q(AB)$ is symmetric).

For positive integers $n$ and $k$ let $B_{n,k}$ be the set of
permutations in $\S_n$ whose Coxeter length is at most $k$. Notice
that $B_{n,k}$ is a union of sets each of which consists of all permutations of a fixed Coxeter length, and so it is set-fine by Proposition~\ref{fine_list}(i) together with
Observation~\ref{union}. Noticing that $\{B_{n,k}B_{n,r}\}=B_{n,k+r}$, we have the following.

\begin{observation}
For every positive integers $k$ and $r$, the set product
$\{B_{n,k}B_{n,r}\}$ is set-fine.
\end{observation}

Note, however, that the multiset product of subsets of
permutations of bounded Coxeter length is not necessarily fine.
For example, letting $A=\{\pi \in \S_4:\ \ell(\pi)\le 1\}$,
the multiset $AA$ is not fine (in fact, $\Q(A)$ is not even symmetric),
and the same holds if we let $A=\{\pi \in \S_5:\ \ell(\pi)\le 2\}$. The next
lemma shows that conjugacy classes behave better with respect to
products.

\begin{lemma}
Multiset and set products of conjugacy classes in $\S_n$ are fine.
\end{lemma}

\begin{proof}
Set products of conjugacy classes are closed under conjugation,
thus, by Proposition~\ref{fine_list}(ii), they are set-fine.
Similarly, multiset products of conjugacy classes are multiset
unions of conjugacy classes, thus by
Proposition~\ref{fine_list}(ii) together with
Observation~\ref{union} they are multiset-fine.
\end{proof}

Conjugacy classes span the center of the group algebra
$\CC[\S_n]$. Inverse descent classes span the Solomon descent
subalgebra, from where the next result follows.

\begin{proposition}\label{conj_des_classes}
Multiset and set products of inverse descent classes in $\S_n$ are fine.
\end{proposition}

\begin{proof}
For every triple $I,J,K$ of subsets of $[n-1]$,
the multiplicity of $D_{n,K}^{-1}$ in the multiset product
$D_{n,I}^{-1} D_{n,J}^{-1}$ equals the the multiplicity of
$D_{n,K}$ in the multiset product $D_{n,J} D_{n,I}$, which
is a nonnegative integer (see 
\cite[Cor. 15]{Gessel}). Thus, both $D_{n,I}^{-1} D_{n,J}^{-1}$
and $\{D_{n,I}^{-1} D_{n,J}^{-1}\}$ are unions of inverse descent
classes (possibly with multiplicities). The result now follows
from Lemma~\ref{lem:DJ}.
\end{proof}

The main theorem in this section is a significant strengthening of
Proposition~\ref{conj_des_classes} in the multiset case.

\begin{theorem}\label{main_FD}
For every fine set $\BBB$ of permutations in $\S_n$ and every
$J\subseteq[n-1]$,
\[
\BBB D_{n,J}^{-1}
\]
is  a fine multiset.
\end{theorem}

In the rest of this section we prove this result, which is part of
Theorem~\ref{main} below. The proof relies on the fact that right
multiplication of a permutation $\pi\in \S_n$
by a union of inverse descent classes $\bigsqcup_{I\subseteq J} D_{n,I}^{-1}$
amounts to shuffling
contiguous segments of appropriate lengths in the sequence
$\pi(1),\dots,\pi(n)$. On the other hand, shuffling
fine sets on disjoint sets of letters may be interpreted as an
algebraic operation on the corresponding
representations, as described in Lemma~\ref{lemma_shuffles} below.
This will be used to show that if $\BBB$ is a fine multiset of
permutations in $\S_n$, then the product of $\BBB$ by the above union of inverse descent classes
is multiset-fine. Finally we deduce that $\BBB D_{n,J}$ is multiset-fine by using
the inclusion-exclusion principle and applying representation-theoretic arguments
to keep track of Schur-positivity.

\subsection{Shuffles of fine sets}\label{subsec:shuffle_fine}

Recall from Subsection~\ref{prel_sec:shuffles} the notation
$A\shuffle B$ for the shuffle of two sets (or multisets) of
permutations, and $(\mu,\nu)$ for the skew Young diagram
consisting of two components of shapes $\mu$ and~$\nu$.

\begin{remark}
When we refer to the group $\S_k\times \S_{n-k}$ in this section,
we consider its natural embedding in $\S_n$, where $\S_k$ permutes
the letters $1,\dots, k$ and $\S_{n-k}$ permutes the letters
$k+1,\dots, n$. Similarly, $\S_{\eta_1}\times \cdots \times
\S_{\eta_m}$ is embedded in $\S_n$, where $\S_{\eta_1}$ permutes
the first $\eta_1$ letters, $\S_{\eta_2}$ permutes the next
$\eta_2$ letters, etc.
\end{remark}

\begin{lemma}\label{lemma_shuffles}
\begin{enumerate}
\item Fix a set partition $U\sqcup V=[n]$ with $|U|=k$. Let
$A$ and $B$ be fine (multi)sets of the symmetric groups on $U$ and
$V$, respectively. Then $A\shuffle B$ is a fine (multi)set of
$\S_n$, and
\begin{equation}\label{eq_shuffles}
\Q(A\shuffle B)=\Q(A)\Q(B).
\end{equation}
In other words, if $A$ is a fine (multi)set for the
$\S_k$-representation $\phi$ and $B$ is a fine (multi)set for the
$\S_{n-k}$-representation $\psi$, then $A\shuffle B$ is a fine
(multi)set for the induced representation $(\phi\otimes
\psi)\uparrow_{\S_k\times \S_{n-k}}^{\S_n}$.
\item More
generally, let $\eta=(\eta_1,\dots,\eta_m)$ be a composition of
$n$ and let $\bigsqcup_{i=1}^m U_i=[n]$ be a set partition with
$|U_i|=\eta_i$. If for all $1\le i\le m$ the set $A_i$ is fine in
the symmetric group on $U_i$ for the $\S_{\eta_i}$-representation
$\phi_i$, then $\shuffle_{i=1}^m A_i$ is a fine set for the
induced representation $(\bigotimes_{i=1}^m
\phi_i)\uparrow_{\S_{\eta_1}\times \cdots \times \S_{\eta_m}}^{\S_n}$.
\end{enumerate}
\end{lemma}

\begin{proof}
To prove part 1, it suffices to prove
Equation (\ref{eq_shuffles}). By Proposition~\ref{Stanley_lemma},
we may assume that $U=\{1,\dots,k\}$ and $V=\{k+1,\dots,n\}$.

By Theorem~\ref{m1} together with 
Remark~\ref{clarification}, $A$ is (multi)set-fine if and only if
there exist a (multi)set partition $A = A_1 \sqcup \ldots \sqcup
A_r$ and $\Des$-preserving bijections from each $A_i$ to a Knuth
class of a suitable shape $\mu^{(i)} \vdash n$, for $1 \le i \le
r$. The same holds for $B$, with (multi)set partition $B = B_1 \sqcup
\ldots \sqcup B_s$ and $\Des$-preserving bijections from each
$B_j$ to a Knuth class of a shape $\nu^{(j)} \vdash n$, for each
$1 \le j \le s$. Then, by Proposition~\ref{G1},
\[
\Q(A)=\sum\limits_i s_{\mu^{(i)}}, \qquad \Q(B)=\sum\limits_j s_{\nu^{(j)}},
\]
and so
\[
\Q(A)\Q(B)=\sum\limits_{i,j} s_{\mu^{(i)}} s_{\nu^{(j)}}.
\]
But by Equation~\eqref{eq:smunu} and Proposition~\ref{G1}, we have
\begin{equation}\label{eq:ij}
s_{\mu^{(i)}} s_{\nu^{(j)}} =
s_{(\mu^{(i)},\nu^{(j)})}=\Q(\SYT((\mu^{(i)},\nu^{(j)})))=\sum\limits_{T\in
\SYT((\mu^{(i)},\nu^{(j)}))} F_{n,\Des(T)}.
\end{equation}
Finally, by Theorem~\ref{BV}, the RHS of
Equation~\eqref{eq:ij} is equal to $\Q(A_i \shuffle B_j)$, where
$A_i$ and $B_j$ are the factors of the above-mentioned (multi)set
partitions of $A$ and $B$. We conclude that
\[
\Q(A)\Q(B)=\sum\limits_{i,j}
\Q(A_i \shuffle B_j) = Q (A\shuffle B). 
\]

 By Remark~\ref{m_shuffles}, the proof of part 1 may be generalized to shuffles of any number of
fine sets on distinct sets of letters, proving part 2.
\end{proof}

\subsection{Fine sets of Young subgroups}

For a composition $\eta=(\eta_1,\dots,\eta_m)$ of $n$, denote by
$\S_\eta\subseteq\S_n$ the corresponding {\em Young subgroup}
$\S_\eta:=\S_{\eta_1}\times \cdots \times \S_{\eta_m}$. Next we
consider the restriction of $\S_n$-representations on fine sets to
Young subgroups.

For compositions $\mu$ of $k$ and $\nu$ of $n-k$ denote by
$\mu,\nu$ their concatenation, which is a composition of $n$. Note
that for every $\pi\in \S_n$, $\Des(\pi) \setminus
S(\mu,\nu)\subseteq [k-1]\cup [k+1,n-1]$.

\begin{defn}\label{Young_character}
A (multi)set of permutations $\BBB$ in $\S_n$ is $\S_k\times
\S_{n-k}$-fine for its complex representation $\rho$ if for each
pair of compositions $\mu$ of $k$ and $\nu$ of $n-k$, the
character value of $\rho$ at a conjugacy class of cycle type
$\mu,\nu$ satisfies
\[
\chi^\rho(\mu,\nu) = \sum\limits_{\pi \in \BBB_{\mu,\nu}}
m(\pi,\BBB)(-1)^{|\Des(\pi) \setminus S(\mu,\nu)|},
\]
where $\BBB_{\mu,\nu}$ is the set of $\mu,\nu$-modal permutations
in $\BBB$ and $m(\pi,\BBB)$ is the multiplicity of $\pi\in \BBB$.
\end{defn}

\begin{remark}
Note that an $\S_k \times \S_{n - k}$-fine set may actually not
be a subset of $\S_k \times \S_{n - k}$. For example, the
singleton $\{3124\}$ which consists of one permutation in $\S_4$
is fine for the $\S_2\times \S_2$-irreducible representation
$S^{(1^2)}\otimes S^{(2)}$.
\end{remark}

\begin{proposition}\label{Young_equidistribution}
A (multi)set of permutations $\BBB$ in $\S_n$ is $\S_k\times
\S_{n-k}$-fine if and only if there exist a (multi)set partition
$\BBB = \BBB_1 \sqcup \ldots \sqcup \,\BBB_m$ and, for each $1 \le
i \le m$, a $\Des$-preserving bijection from $\BBB_i$ to the set of
all pairs of standard Young tableaux $(T^1,T^2)$ of suitable
shapes $\lambda^{(i)}\vdash k$ and $\mu^{(i)}\vdash n-k$,
respectively, where $T^1$ consists of the letters $1,\dots, k$,
$T^2$ consists of the letters $k+1,\dots, n$, and
$\Des(T^1,T^2):=\Des(T^1)\sqcup \Des(T^2)$.
\end{proposition}

The proof is similar to the proof of~\cite[Theorem 1.5]{Adin-R} and
is omitted.

\begin{lemma}\label{lemma_restriction}
If $\BBB\subseteq \S_n$ is a fine (multi)set for the
$\S_n$-representation $\rho$, then for every $1\le k\le n$, $\BBB$
is a fine (multi)set of $\S_k\times \S_{n-k}$ for the restricted
representation $\rho\downarrow^{\S_n}_{\S_k\times \S_{n-k}}$.
\end{lemma}

\begin{proof} Consider the natural embedding of $\S_k\times
\S_{n-k}$ in $\S_n$, where $\S_k$ permutes the letters $1,\dots,
k$ and $\S_{n-k}$ permutes the letters $k+1,\dots, n$. By
definition, for every $\sigma\in \S_k$ of cycle type $\mu$ and
every $\tau\in \S_{n-k}$ of cycle type $\nu$,
\[
\chi^{\rho\downarrow_{\S_k\times \S_{n-k}}}(\sigma,\tau)=
\chi^\rho(\sigma\tau)=\sum\limits_{\pi \in \BBB_{\mu,\nu}}
(-1)^{|\Des(\pi) \setminus S(\mu,\nu)|},
\]
and so $\BBB$ is a fine (multi)set for the restricted
representation $\rho\downarrow_{\S_k\times \S_{n-k}}$.
\end{proof}

\begin{remark}\label{m_shuffles2}
All statements in this subsection may be easily generalized from
the Young subgroup $\S_k\times \S_{n-k}$ to any Young
subgroup $\S_\eta$ of $\S_n$.
\end{remark}

\subsection{Right multiplication by an inverse descent class}

Now we are ready to prove an extended version of
Theorem~\ref{main_FD}. We use the notation $\{j_1,\dots,j_t\}_<$ to indicate that the elements of the set
satisfy $j_1< j_2< \cdots <j_t$.

For $J=\{j_1,\dots,j_t\}_<\subseteq [n-1]$, let $\S_{\bar J}$ denote the Young subgroup
$\S_{j_1}\times \S_{j_2-j_1}\times \cdots \times \S_{n-j_t}$. Let
\begin{equation}\label{eq:defRnJ}
R_{n,J}:=\{\pi \in \S_n:\ \Des(\pi)\subseteq
J\}=\bigsqcup_{I\subseteq J} D_{n,I},
\end{equation}
and let $R_{n,J}^{-1}:=\{\pi^{-1}: \pi \in R_{n,J}\}$.

\begin{theorem}\label{main}
Let $\BBB \subseteq \S_n$ be a fine set for the
$\S_n$-representation $\rho$. Then for every $J\subseteq [n-1]$, the following hold.
\begin{enumerate}
\item The multiset
\[
\BBB R_{n,J}^{-1}
\]
is a fine multiset of $\S_n$ for $(\rho\downarrow_{\S_{\bar
J}})\uparrow^{\S_n}$.

\item The multiset
\[
\BBB D_{n,J}^{-1}
\]
is a fine multiset of $\S_n$ for $\rho\otimes S^{Z_{n,J}}$, where
$Z_{n,J}$ is the ribbon as in
Definition~\ref{zigzag_subset}, and $S^{Z_{n,J}}$ is the
corresponding Specht module.
\end{enumerate}
\end{theorem}

The {\em Kronecker product} of two symmetric functions $f,h\in \Lambda^n$
is defined by
\[
f*h:=\sum\limits_{\mu,\nu,\lambda\ \vdash n} \langle f,
s_\mu\rangle \langle h, s_\nu\rangle g_{\mu, \nu, \lambda}
s_\lambda,
\]
where
\[
g_{\mu, \nu, \lambda}:= \langle \chi^\mu\otimes \chi^\nu,
\chi^\lambda\rangle.
\]
With this definition, part 2 in Theorem~\ref{main} is equivalent to the fact that
\[
\Q(\BBB D_{n,J}^{-1})=\Q(\BBB) * \Q(D_{n,J}^{-1}).
\]

\begin{proof}[Proof of Theorem~\ref{main}] For the sake of clarity, let us first prove part 1 for a singleton $J=\{k\}$.

First, notice that $R_{n,\{k\}}^{-1}$ consists of all shuffles of
$1,\dots,k$ with $k+1,\dots,n$. Hence $\BBB R_{n,\{k\}}^{-1}$
consists of all shuffles of $(\pi(1),\dots,\pi(k))$ with
$(\pi(k+1),\dots,\pi(n))$ over all $\pi\in \BBB$.

Since $\BBB$ is a fine set, it follows from
Lemma~\ref{lemma_restriction} that it is also a fine set of
$\S_k\times \S_{n-k}$ for the restricted representation
$\rho\downarrow_{\S_k\times \S_{n-k}}$.
Proposition~\ref{Young_equidistribution} and 
Remark~\ref{clarification} now imply that there exist a multiset
partition $\BBB = \BBB_1 \sqcup \ldots \sqcup \BBB_m$ and
$\Des$-preserving bijections, for each $1 \le i \le m$, from
$\BBB_i$ to the cartesian product $K^{(i)}\times L^{(i)}$ of a
pair of Knuth classes of suitable shapes
$\lambda^{(i)}\vdash k$ and $\mu^{(i)}\vdash n-k$, where $K^{(i)}$
is on the letters $1,\dots, k$ and $L^{(i)}$ is on the letters
$k+1,\dots, n$. Together with Proposition~\ref{Stanley_lemma},
this implies that the distribution of the descent set over $\BBB
R_{n,\{k\}}^{-1}$, that is over all shuffles of
$(\pi(1),\dots,\pi(k))$ with $(\pi(k+1),\dots,\pi(n))$ for all
$\pi\in \BBB$ is equal to its distribution over the union of all
shuffles $K^{(i)}\shuffle L^{(i)}$ for $1\le i\le m$.
Lemma~\ref{lemma_shuffles} and Proposition~\ref{fine_list}(iii)
complete the proof of the first part for $J=\{k\}$.

By Remarks~\ref{m_shuffles} and~\ref{m_shuffles2}, and part 2 of
Lemma~\ref{lemma_shuffles}, this proof easily generalizes to any
subset $J=\{j_1,\dots,j_t\}\subseteq [n-1]$. Now $\BBB
R_{n,J}^{-1}$ consists of all shuffles of
$\pi(1),\dots,\pi(j_1)$; $\pi(j_1+1),\dots,\pi(j_2)$;
$\dots$; $\pi(j_t+1),\dots,\pi(n)$ over all $\pi\in \BBB$,\
and pairs of Knuth classes and shapes are replaced by
$t+1$-tuples.

\medskip

Next we prove part 2. By part 1, $\Q(\BBB R_{n,J}^{-1})$ is
symmetric and Schur-positive. Since $\Q(\BBB
R_{n,J}^{-1})=\sum_{I\subseteq J} \Q(\BBB D_{n,I}^{-1})$, it
follows by the inclusion-exclusion principle that
\begin{equation}
\label{eq:QBD}
\Q(\BBB D_{n,J}^{-1})
=\sum\limits_{I\subseteq J} (-1)^{|J\setminus I|} \Q(\BBB
R_{n,I}^{-1}).
\end{equation}
We conclude that $\Q(\BBB D_{n,J}^{-1})$ is symmetric.

To prove that it is also Schur-positive, recall the following
reciprocity result~\cite[Ch. 3.3 Ex. 5]{Serre}: if $H$ is a subgroup of
a group $G$, and $\psi$ and $\phi$ are representations of $G$ and $H$, respectively, then
\begin{equation}\label{eq:reciprocity}
(\phi\otimes \psi\downarrow_H)\uparrow^G \ \cong\ \phi
\uparrow^G\otimes \psi .
\end{equation}
Hence, in our setting,
\[
(\rho\downarrow_{\S_{\bar J}})\uparrow^{\S_n}\ \cong\ (1_{\S_{\bar
J}}\otimes \rho\downarrow_{\S_{\bar J}})\uparrow^{\S_n}\ \cong\
1_{\S_{\bar J}}\uparrow^{\S_n}\otimes \rho.
\]
It follows from Equation~\eqref{eq:QBD} that $\Q(\BBB
D_{n,J}^{-1})$ is the Frobenius image of
the representation
\[
\sum\limits_{I\subseteq J} (-1)^{|J\setminus
I|}(\rho\downarrow_{\S_{\bar I}})\uparrow^{\S_n} \ \cong\
\sum\limits_{I\subseteq J} (-1)^{|J\setminus I|} 1_{\S_{\bar
I}}\uparrow^{\S_n}\otimes \rho \ \cong\ \rho \otimes
\sum\limits_{I\subseteq J} (-1)^{|J\setminus I|} 1_{\S_{\bar
I}}\uparrow^{\S_n}.
\]
Recalling that the Frobenius image of the induced representation
$1_{\S_{\bar I}}\uparrow^{\S_n}$ is the corresponding homogenous
symmetric function~\cite[Corollary 7.18.3]{Stanley_ECII}, it
follows from~\cite[p. 295]{Gessel} that
\begin{equation}\label{eq:altsum}
\ch\left(\sum\limits_{I\subseteq J} (-1)^{|J\setminus I|} 1_{\S_{\bar
I}}\uparrow^{\S_n}\right)= s_{Z_{n,J}}=\Q(
D_{n,J}^{-1}).
\end{equation}
See also~\cite[Theorem 2]{Solomon} and~\cite[Theorem
1.2]{Stanley}. It follows that $\Q(\BBB D_{n,J}^{-1})$ is a
symmetric function corresponding to a tensor product of two
non-virtual representations and hence Schur-positive.
\end{proof}

\section{Vertical rotations}\label{sec:vertical}

In this section we explore some applications of Theorem~\ref{main} with a geometric flavor.
We use the term {\em vertical rotation} of $\pi\in\S_n$ to mean a permutation obtained by replacing each entry $\pi(i)$ with $\pi(i)+k\bmod n$ for some fixed $k$.

\subsection{Vertical rotations of inverse descent classes}

If we choose the fine set $\BBB$ to be a colayered grid class $\LLL^k_n$,
which is fine by Corollary~\ref{cor:LLLfine}, a consequence of
Theorem~\ref{main} is that $\LLL^k_n D_{n,J}^{-1}$ is a fine
multiset of $\S_n$ for every $k\ge1$ and every $J\subseteq[n-1]$.

Recall the notation $C_n=\langle c\rangle=\{c^k:\ 0\le k< n\}$,
where $c$ is the $n$-cycle $(1,2,\dots,n)$. Note that
$$C_n=\LLL^2_n=\G_n\left(\begin{array}{cc}
           1 & 0 \\
           0 & 1
         \end{array}\right).$$
For $A\subseteq \S_n$, the product $C_n A$ is the multiset of
vertical rotations of the elements of $A$.
The following is a consequence of Theorem~\ref{main}.

\begin{corollary}\label{cor:vertical}
For every $J\subseteq[n-1]$, the multiset $C_n D_{n,J}^{-1}$ of
vertical rotations of an inverse descent class is a fine multiset
for $S^{Z_{n,J}}\downarrow_{\S_{n-1}}\uparrow^{\S_n}$.
\end{corollary}

\begin{proof}
By Corollary~\ref{cor:LLLfine} together with Equation
\eqref{Pierri_1}, $C_n$ is a fine set with
$\Q(C_n)=s_n+s_{n-1,1}=s_1s_{n-1}=\ch(1_{\S_{n-1}}\uparrow^{\S_n})$.
By Lemma~\ref{lem:DJ},
$\Q(D_{n,J}^{-1})=\ch(S^{Z_{n,J}})$. Thus, by Theorem~\ref{main}
and Equation~\eqref{eq:reciprocity}, $C_n D_{n,J}^{-1}$ is  a fine
multiset for
\[
1_{\S_{n-1}}\uparrow^{\S_n}\otimes S^{Z_{n,J}}\ \cong \
(1_{\S_{n-1}}\otimes
S^{Z_{n,J}}\downarrow_{\S_{n-1}})\uparrow^{\S_n}\ \cong \S^{Z_{n,J}}\downarrow_{\S_{n-1}}\uparrow^{\S_n}.
\]
\end{proof}

\subsection{Vertical rotations of grids}\label{vertical_grids}

By Proposition~\ref{one-column_descent_classes}, every one-column grid class $\G^\vv_n$ is a disjoint union of inverse descent classes. Thus, we get the following consequence of Theorem~\ref{main}.

\begin{corollary}\label{cor:LLLGv}
For every one-column grid class $\G^\vv$ and every $k\ge1$,
\[
\LLL^k_n \G^\vv_n
\]
is a fine multiset of $\S_n$.
\end{corollary}

Taking $k=2$, Corollary~\ref{cor:LLLGv} implies that the multiset of vertical rotations $C_n \G^\vv_n$ is fine.
For an arbitrary $\vv$, we do not know if the underlying set is always fine, see Conjecture~\ref{conj:rotations-onecolumn}. However, we will show that this is the case sometimes.

Consider now the grid class $\G^{+^k}$, whose elements are shuffles of $k$ increasing sequences.
By the above paragraph, $C_n\G^{+^k}_n$ is a fine multiset. The underlying set
$\{C_n\G^{+^k}_n\}$ is the grid class $\G_n(M_k)$, where $M_k$ is
the $2k\times 2$ matrix whose odd-numbered rows are $(1,0)$ and whose even-numbered rows are $(0,1)$. The grid $\G(M_3)$ is drawn in
Figure~\ref{fig:GM3}. We will show in Section~\ref{sec:horizontal} that $\G_n(M_k)$ is a fine set. In the rest of this subsection we make some initial steps towards this goal.

\begin{figure}[htb]
\centering
\begin{tikzpicture}[scale=.6]
\draw (0,0) rectangle (6,6); \draw[dotted] (0,2)--(6,2);
\draw[dotted] (0,4)--(6,4); \draw[thick] (0,0)--(6,2);
\draw[thick] (0,2)--(6,4); \draw[thick] (0,4)--(6,6);
\end{tikzpicture}
\hspace{10mm}
\begin{tikzpicture}[scale=.6]
\draw (0,0) rectangle (6,6); \draw[dotted] (0,1)--(6,1);
\draw[dotted] (0,3)--(6,3); \draw[dotted] (0,5)--(6,5);
\draw[dotted] (3,0)--(3,6);
\draw[dotted,very thin] (0,2)--(6,2); \draw[dotted,very thin] (0,4)--(6,4);
\draw[thick] (3,0)--(6,1);
\draw[thick] (0,1)--(6,3); \draw[thick] (0,3)--(6,5); \draw[thick]
(0,5)--(3,6);
\end{tikzpicture}
\caption{The grids $\G^{+^3}$ (left) and $\G(M_3)$ (right).}
\label{fig:GM3}
\end{figure}

Cyclic descents were introduced by Cellini~\cite{Cellini} and
further studied in~\cite{Petersen, Dilks}.

\begin{defn}\label{def_cyc_des}
The {\em cyclic descent set} of a permutation $\pi\in \S_n$ is
$$\cDes(\pi)=\begin{cases} \Des(\pi) & \text{if }\pi(n)<\pi(1),\\
 \Des(\pi)\cup\{n\} & \text{if }\pi(n)>\pi(1).\end{cases}
$$
The {\em cyclic descent number} is $\cdes(\pi):=|\cDes(\pi)|$.
\end{defn}

The next lemma will be useful here and in Section~\ref{sec:horizontal}. Recall the notation $c=(1,2,\dots,n)\in C_n$.

\begin{lemma}\label{lem:cdes_rotations}
For every $\sigma\in\S_n$ and $j$, we have $\cdes(\sigma)=\cdes(\sigma c^j)=\cdes(c^j \sigma)$.
\end{lemma}

\begin{proof}
It is clear by definition that horizontal rotations preserve $\cdes$, and so $\cdes(\sigma)=\cdes(\sigma c^j)$. For vertical rotations, note that to obtain $c\pi$ from a permutation $\pi\in\S_n$ we add $1$ from the values of all the entries, except that the entry $n$ becomes $1$. If $i$ is the position of $n$ in $\pi$, so that $\pi(i)=n$ and $c\pi(i)=1$, then the elements of $\cDes(\pi)$ and $\cDes(c\pi)$ other than $i-1,i$ are the same (where we are defining $i-1=n$ if $i=1$). On the other hand, $\cDes(\pi)\cap\{i-1,i\}=\{i\}$ and $\cDes(c\pi)\cap\{i-1,i\}=\{i-1\}$. It follows that $\cdes(\pi)=\cdes(c\pi)$, and so, by iteration, $\cdes(\sigma)=\cdes(c^j \sigma)$.
\end{proof}

Recall from Observation~\ref{obs:Gk} that
$\G^{+^k}_n=\{\pi\in\S_n: \des(\pi^{-1})\le k-1\}$. One can
characterize $\G_n(M_k)$ similarly as follows.

\begin{lemma}\label{lem:GMkcdes} For every $k\ge1$,
$$\G_n(M_k)=\{\pi\in\S_n:\cdes(\pi^{-1})\le k\}.$$
In particular, by Lemma~\ref{lem:cdes_rotations}, $\G_n(M_k)$ is closed under vertical and horizontal rotations.
\end{lemma}

\begin{proof}
By construction of the grid for $\G(M_k)$, we have that $\pi\in\G_n(M_k)$ if and only if $\pi$ is a vertical rotation of some $\sigma\in\G^{+^k}_n$, that is, $\pi=c^j\sigma$ for some $j$ and some $\sigma$ with $\des(\sigma^{-1})\le k-1$. By Lemma~\ref{lem:cdes_rotations}, $\cdes(\pi^{-1})=\cdes(\sigma^{-1}c^{-j})=\cdes(\sigma^{-1})\le\des(\sigma^{-1})+1\le k$.

For the converse, suppose that $\cdes(\pi^{-1})\le k$. By rotating $\pi$ vertically until its rightmost entry is $n$,
we can write $\pi=c^j\sigma$ for some $j$ and some $\sigma$ with $\sigma(n)=n$. Now, since $\cdes(\sigma^{-1})=\cdes(\pi^{-1})\le k$ and $\des(\sigma^{-1})=\cdes(\sigma^{-1})-1$,
we have that $\des(\sigma^{-1})\le k-1$, and so  $\sigma\in\G^{+^k}_n$.
\end{proof}

It follows from Lemma~\ref{lem:GMkcdes} and  Observation~\ref{obs:Gk}, or directly by looking at the drawings of the grids, that
$\G(M_{k-1})\subset\G^{+^k}\subset\G(M_k)$.

\begin{proposition}\label{prop:rotated_shuffles}
For every $k\ge1$, the quasisymmetric function $\Q(\G_n(M_k))$
is symmetric.
\end{proposition}

\begin{proof}
As mentioned after Corollary~\ref{cor:LLLGv}, $C_n\G^{+^k}_n$ is a fine multiset, and
$\G_n(M_k)=\{C_n\G^{+^k}_n\}$. Our proof is by induction on $k$.
For $k=1$, $\G_n(M_1)=C_n=\LLL^2_n$;
thus, by Corollary~\ref{cor:LLLfine},
\begin{equation}\label{eq:M1}\Q(\G_n(M_1))=\Q(\LLL^2_n)=s_n+s_{n-1,1}.\end{equation}

For $k\ge2$, we look at the multiplicities of the elements in
$C_n\G^{+^k}_n$. Take $\pi\in\G^{+^k}_n$, and consider two cases:
\begin{itemize}
\item If $\pi\in \G_n(M_{k-1})$, then its orbit $C_n\pi$ is contained in $\G_n(M_{k-1})$ by Lemma~\ref{lem:GMkcdes},
hence in $\G^{+^k}_n$. Thus, since each of the $n$ permutations in $C_n\pi$ generate the same orbit,
the permutations in this orbit appear with multiplicity $n$ in $C_n\G^{+^k}_n$.

\item If $\pi\in\G^{+^k}_n\setminus\G_n(M_{k-1})$ (equivalently, $\des(\pi^{-1})=k-1$ and
$\cdes(\pi^{-1})=k$, by Lemma~\ref{lem:GMkcdes}), then its orbit $C_n\pi$ contains exactly $k$
elements of $\G^{+^k}_n$. Indeed, these are elements of the form $c^j\pi$, with $0\le j<n$, satisfying $\des((c^j\pi)^{-1})\le k-1$, which, using Lemma~\ref{lem:cdes_rotations}, is equivalent to the condition $\des(\pi^{-1}c^{-j})\le \cdes(\pi^{-1}c^{-j})-1$. But this happens precisely when $n\in\cDes(\pi^{-1}c^{-j})$, which holds for exactly $k$ values of $j$.
Since each of the $k$ elements of $C_n\pi\cap\G^{+^k}_n$ generate the same orbit $C_n\pi$, the permutations in this orbit
appear with multiplicity $k$ in $C_n\G^{+^k}_n$.
\end{itemize}
It follows that $$\Q(C_n\G^{+^k}_n)=n\,
\Q(\G_n(M_{k-1}))+k\,\Q(\G_n(M_{k})\setminus\G_n(M_{k-1})),$$ and
so
\begin{equation}\label{eq:recurrence}
k\,\Q(\G_n(M_{k}))=\Q(C_n\G^{+^k}_n)-(n-k)\,\Q(\G_n(M_{k-1})).
\end{equation}
We conclude that $\Q(\G_n(M_k))$ is symmetric.
\end{proof}

For $k=2$, the above argument can be used to show that $\Q(\G_n(M_2))$ is Schur-positive.
In Section~\ref{sec:horizontal}, it will be proved that $\Q(\G_n(M_k))$ is Schur-positive for all $k$.

\begin{proposition}\label{R2}
$\G_n(M_2)$ is a fine set, and
\begin{multline*}\Q(\G_n(M_2))=s_n+(n-1)s_{n-1,1}+\sum_{a=2}^{\lfloor\frac{n}{2}\rfloor-1}(2n-4a+2)s_{n-a,a}+\sum_{a=1}^{\lfloor \frac{n-1}{2}\rfloor}(n-2a)s_{n-a-1,a,1}\\
+\begin{cases} 2s_{\frac{n}{2},\frac{n}{2}} & \text{for even }n\ge4, \\ 4s_{\frac{n+1}{2},\frac{n-1}{2}}& \text{for odd }n\ge5, \\
0 & \text{for }n\le 3. \end{cases}
\end{multline*}
\end{proposition}

\begin{proof}
Using Equations~\eqref{eq:M1} and~\eqref{eq:recurrence} for $k=2$, we have
\begin{equation}\label{eq:GM2}
\Q(\G_n(M_2))=\frac{1}{2}\left(\Q(C_n\G^{+^2}_n)-(n-2)(s_n+s_{n-1,1})\right).
\end{equation}
Since $\Sh_n$ is a union of descent classes,
Corollaries~\ref{cor:vertical} and~\ref{cor:QSh} imply that
\[
\Q(C_n\G_n^{+^2})=\ch\left(\chi^{(n)}\downarrow_{\S_{n-1}}\uparrow^{\S_n}+\sum_{a=1}^{\lfloor\frac{n}{2}\rfloor}
(n-2a+1)\chi^{(n-a,a)}\downarrow_{\S_{n-1}}\uparrow^{\S_n}\right).
\]
Counting all the ways to remove and add a box in a two-row shape, we obtain
\begin{multline*}
\Q(C_n\G_n^{+^2})=ns_n+(3n-4)s_{n-1,1}+\sum_{a=2}^{\lfloor\frac{n}{2}\rfloor-1}(4n-8a+4)s_{n-a,a}+\sum_{a=1}^{\lfloor \frac{n-1}{2}\rfloor}(2n-4a)s_{n-a-1,a,1}\\
+\begin{cases} 4s_{\frac{n}{2},\frac{n}{2}} & \text{for even }n\ge4, \\ 8s_{\frac{n+1}{2},\frac{n-1}{2}}& \text{for odd }n\ge5, \\
0 & \text{for }n\le 3. \end{cases}
\end{multline*}
The formula for $\Q(\G_n(M_2))$ follows now from
Equation~\eqref{eq:GM2}.
\end{proof}

\subsection{Arc permutations revisited}\label{sec:arc_revisited}

In general, for an arbitrary one-column grid class $\G^\vv$, the set $\{C_n\G^\vv\}$ of its
vertical rotations may not be a grid class, but it is a union of grid classes. For example, taking the class $\LL=\G^{-+}$ of left-unimodal permutations, we obtain the following.

\begin{observation}\label{obs:arc_revisited}
Arc permutations are obtained as vertical rotations of left-unimodal permutations, that is,
$$\A_n=\{C_n \LL_n\},$$
and so they can be described as a union of grid classes, namely the two drawn in Figure~\ref{fig:grid_arc}:
$$\A_n=\G_n\left(\begin{array}{cc}
           1 & 0 \\
           -1 & 0 \\
           0 & -1 \\
           0 & 1 \\
         \end{array}
       \right)\ \cup\
\G_n\left(\begin{array}{cc}
           0 & -1 \\
           0 & 1 \\
           1 & 0 \\
           -1 & 0 \\
         \end{array}
       \right).$$
\end{observation}

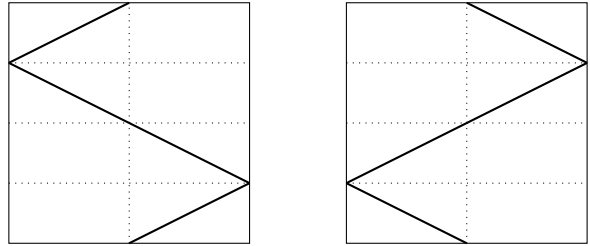
\begin{figure}[htb]
\centering
\begin{tikzpicture}[scale=.8]
\draw (0,0) rectangle (4,4); \draw[dotted] (0,1)--(4,1);
\draw[dotted] (0,2)--(4,2); \draw[dotted] (0,3)--(4,3);
\draw[dotted] (2,0)--(2,4); \draw[thick] (2,4)--(0,3);
\draw[thick] (4,1)--(0,3); \draw[thick] (2,0)--(4,1);
\end{tikzpicture}
\hspace{10mm}
\begin{tikzpicture}[scale=.8]
\draw (0,0) rectangle (4,4); \draw[dotted] (0,1)--(4,1);
\draw[dotted] (0,2)--(4,2); \draw[dotted] (0,3)--(4,3);
\draw[dotted] (2,0)--(2,4); \draw[thick] (2,4)--(4,3);
\draw[thick] (0,1)--(4,3); \draw[thick] (2,0)--(0,1);
\end{tikzpicture}
\caption{\label{fig:grid_arc} Grids for arc permutations.}
\end{figure}

The following result reformulates~\cite[Theorem 5]{ER1}. Here we
give a short proof, as an immediate consequence of our approach.

\begin{proposition}
\label{prop:rotated_leftunimodal}
$\A_n$ is a fine set, and
\begin{equation}\label{eq_arc_schur}
\Q(\A_n)=s_{n}+s_{1^n}+\sum\limits_{k=2}^{n-2}s_{n-k,2,1^{k-2}}+2\sum\limits_{k=1}^{n-2}s_{n-k,1^k}.
\end{equation}
\end{proposition}

\begin{proof}
By Corollary~\ref{cor:vertical} together with Lemma~\ref{lem:QLL},
$C_n\LL_n$ is a fine multiset for the $\S_n$-character
\[
\sum\limits_{k=0}^{n-1}\chi^{(n-k,1^k)}\downarrow_{\S_{n-1}}\uparrow^{\S_n}
=\left(2\sum\limits_{k=0}^{n-2}\chi^{(n-k-1,1^k)}\right)\uparrow^{\S_n}
=2(\chi^{(n)}+\chi^{(1^n)})+2\sum\limits_{k=2}^{n-2}\chi^{(n-k,2,1^{k-2})}
+4\sum\limits_{k=1}^{n-2}\chi^{(n-k,1^k)}.
\]
To show that the underlying $\A_n=\{C_n \LL_n\}$ set is fine, we count multiplicities of permutations in the multiset.
For every $\pi\in\LL_n$, exactly two elements of the orbit
$C_n\pi$ belong to $\LL_n$, namely $\pi$ and either $c\pi$ or
$c^{-1}\pi$, depending on whether $\pi(n)=n$ or $\pi(n)=1$,
respectively. Thus, every element in $C_n\LL_n$ appears with
multiplicity two, and so
\[
\Q(\A_n)=\Q(\{C_n\LL_n\})=\frac{1}{2}\Q(C_n\LL_n)=s_{n}+s_{1^n}+\sum\limits_{k=2}^{n-2}s_{n-k,2,1^{k-2}}+2\sum\limits_{k=1}^{n-2}s_{n-k,1^k}.\qedhere
\]
\end{proof}

\begin{remark}\label{rem:arc} The set $$\G_n\left(\begin{array}{cc}
           1 & 0 \\
           -1 & 0 \\
           0 & -1 \\
           0 & 1
         \end{array}\right),$$
which is one of the two grid classes for arc permutations, is not
a fine set (or even symmetric) already for $n=4$.
\end{remark}

\section{Horizontal rotations}\label{sec:horizontal}

In this section
we identify $\S_{n-1}$ with the subset of $\S_n$ consisting of those permutations $\pi$ with $\pi(n)=n$.
In particular, subsets of $\S_{n-1}$ such as $D_{n-1,J}$ and $R_{n-1,J}$ will be considered as subsets of $\S_n$ as well.

While Section~\ref{sec:vertical} discussed vertical rotations of inverse descent classes and grid classes as an application of Theorem~\ref{main}, this section deals with horizontal rotations. Two important differences are that the rotated sets are now in $\S_{n-1}$, and that the proofs do not rely on Theorem~\ref{main} but instead use bijective techniques. In Section~\ref{sec:proof_horizontal} we prove the following result, and in Section~\ref{Section7:cor} we use it to derive a significant strengthening of Proposition~\ref{prop:rotated_shuffles}, by showing that $\G_n(M_k)$ is fine.

\begin{theorem}\label{thm:horizontal1}
For every $J\subseteq [n-2]$,
$D_{n-1,J}^{-1}C_n$ is a fine set for
$S^{Z_{n-1,J}}\uparrow^{\S_n}$.
\end{theorem}

\begin{remark}\label{rem:Cn}
For every $A\subseteq \S_{n-1}$, the multiset $A C_n$ is in fact a set. For example,
$D_{n-1,J}^{-1} C_n=\{D_{n-1,J}^{-1} C_n\}$.
In particular, since $\S_{n-1} C_n=\S_n$, every $\pi\in\S_n$ can be written uniquely as $\pi=\sigma c^j$ where $\sigma\in\S_{n-1}$ and $0\le k<n$.
\end{remark}

\subsection{Horizontal rotations of inverse descent classes in $\S_{n-1}$}\label{sec:proof_horizontal}

For $J=\{j_1,\dots,j_t\}_<\subseteq [n-2]$, let $L_{n,J}$ be the skew shape of size $n$
consisting of disconnected horizontal strips of sizes $j_1, j_2-j_1,  \dots , n-1-j_t, 1$ from left to right, each strip touching the next one at a single point. Let $T$ be a SYT of shape $L_{n,J}$ and let $T_1$ the
entry in its upper-right box.

\begin{ex} The skew SYT
\[
T=\young(::::4,:135,2)
\]
has shape $L_{5,\{1\}}$ and $T_1=4$.
\end{ex}

Recall the definition of $R_{n,J}$ from Equation~\eqref{eq:defRnJ}.

\begin{lemma}\label{lemma:horizontal2}
For every $J\subseteq [n-2]$, $I\subseteq [n-1]$, and $1\le k\le n$,
\[
|\{\pi\in R_{n-1,J}^{-1}C_n:\ \pi^{-1}(n)=k, \ \Des(\pi)=I\}|=
|\{T\in \SYT(L_{n,J}):\ T_1=k, \ \Des(T)=I\}|.
\]
\end{lemma}

\begin{proof}
We will construct a $\Des$-preserving bijection from $R_{n-1,J}^{-1}C_n$ to $\SYT(L_{n,J})$,
under which the position of $n$ in the permutation becomes $T_1$
in the SYT. In this proof, addition and subtraction will be done
modulo $n$, so that the entries of the tableaux are always
between $1$ and $n$.

For each $\sigma\in R_{n-1,J}^{-1}\subset \S_n$, let $T^\sigma$ be
the SYT of shape $L_{n,J}$ whose entries from left to right are
$\sigma^{-1}(1),\sigma^{-1}(2),\dots,\sigma^{-1}(n)=n$. Let
$T^\sigma+k$ be the tableau obtained from $T^\sigma$ by adding $k$
to each cell (and reducing modulo $n$ if necessary).

Given $\pi=\sigma c^{-k}\in R_{n-1,J}^{-1}C_n$, where $0\le k< n$,
let $T^\pi\in \SYT(L_{n,J})$ be obtained from $T^\sigma+k$ by
rearranging the entries in each row in increasing order, so that
it becomes a SYT.

To show that the map $\pi\mapsto T^\pi$ is a bijection, note that
given $T\in\SYT(L_{n,J})$ one can recover $\pi$ such that $T^\pi=T$ as follows.
Letting $k=T_1$, subtract $k$ from each entry of $T$ and then
rearrange the entries in each row in increasing order. Let
$\sigma$ be the inverse of the permutation obtained by reading the
resulting SYT from left to right, and let $\pi=\sigma c^{-k}$. Note that $\sigma\in R_{n-1,J}^{-1}$.

It is immediate from the construction of $T^\pi$ that its rightmost entry equals the position of $n$ in $\pi$,
namely, $T^\pi_1=\pi^{-1}(n)$.

It remains to prove that $\Des(T^\pi)=\Des(\pi)$. Note that
$i\in\Des(\pi)$ if and only if $i+1$ appears to the left of $i$ in
the sequence $\pi^{-1}(1),\dots,\pi^{-1}(n)$. If $\pi=\sigma
c^{-k}$, then $\pi^{-1}=c^k \sigma^{-1}$, so
$\pi^{-1}(1),\dots,\pi^{-1}(n)$ is the list of entries of
$T^\sigma+k$ from left to right. Since $i\in\Des(T^\pi)$ if and
only if $i+1$ appears to the left of $i$ in $T^\pi$, it only
remains to show that $i+1$ appears to the left of $i$ in
$T^\sigma+k$ if and only if the same is true in $T^\pi$. The only
way for this to fail would be if $i+1$ and $i$ are in the same row
of $T^\sigma+k$, with $i+1$ to the left of $i$, so that
rearranging the row in increasing order changes their relative
position. Since the rows of $T^\sigma$ are increasing, this is
only possible if the corresponding entries in $T^\sigma$, which are $i+1-k$ and $i-k$ (mod $n$), equal $1$ and $n$, respectively. But this contradicts the fact that $n$ is in a
row by itself in $T^\sigma$.
\end{proof}

\begin{ex}
Let $n=5$, $J=\{1\}\subseteq[3]$ and $\sigma=2314\in R_{\{1\}, 4}^{-1}\subset\S_5$, so $\sigma^{-1}=3124$. Then
\[
\sigma C_5=
\{\sigma c^0,\sigma c^{-1},\sigma c^{-2},\sigma c^{-3},\sigma c^{-4}\}
=\{23145,52314,45231,14523,31452\}.
\]
The descent sets  of these horizontal rotations are $\{2\},\ \{1,3\},\ \{2,4\},\ \{3\},\
\{1,4\}$, respectively.
 The SYT corresponding to these permutations are
\[
T^{23145}=\young(::::5,:124,3),\ T^{52314}=\young(::::1,:235,4),\
T^{45231}=\young(::::2,:134,5),
\]
\[
\ T^{14523}=\young(::::3,:245,1),\ T^{31452}=\young(::::4,:135,2),
\]
whose descent sets are
 $\{2\},\ \{1,3\},\ \{2,4\},\ \{3\},\
\{1,4\}$, respectively.
\end{ex}

For $J=\{j_1,\dots,j_t\}_<\subseteq [n-2]$, recall the notation $\S_{\bar J}:=\S_{j_1}\times
\S_{j_2-j_1}\times \cdots \times \S_{n-1-j_t}\subseteq\S_{n-1}$.

\begin{corollary}\label{cor:horizontal3}
For every $J\subseteq [n-2]$,
$R_{n-1,J}^{-1}C_n$ is a fine set for $1_{S_{\bar
J}}\uparrow^{\S_n}$.
\end{corollary}

\begin{proof}
By Lemma~\ref{lemma:horizontal2}, the distribution of $\Des$ over $R_{n-1,J}^{-1}C_n$ is the same as over $\SYT(L_{n,J})$. By Lemma~\ref{eq_skew_induced} together with
Proposition~\ref{G1}, the set $\SYT(L_{n,J})$ is fine
for the induced character $1_{\S_{\bar J}}\uparrow^{\S_n}$, and thus so is $R_{n-1,J}^{-1}C_n$.
\end{proof}

Now we are ready to prove Theorem~\ref{thm:horizontal1}.

\begin{proof}[Proof of Theorem~\ref{thm:horizontal1}]
Since $R_{n,J}^{-1}=\bigsqcup_{I\subseteq J} D_{n,I}^{-1}$, we have that
$$\Q(R_{n,J}^{-1}C_n)=\sum_{I\subseteq J} \Q(D_{n,I}^{-1}C_n).$$
By the inclusion-exclusion principle,
\[
\Q(D_{n,J}^{-1}C_n)=\sum\limits_{I\subseteq J} (-1)^{|J\setminus
I|} \Q( R_{n,I}^{-1} C_n).
\]
Hence, by Corollary~\ref{cor:horizontal3}, $\Q(D_{n,J}^{-1}C_n)$
is symmetric, and in fact it
is the Frobenius image of the alternating sum of representations
\[
\sum\limits_{I\subseteq J} (-1)^{|J\setminus I|} 1_{\S_{\bar
J}}\uparrow^{\S_n}=\sum\limits_{I\subseteq J} (-1)^{|J\setminus
I|} 1_{\S_{\bar J}}\uparrow^{\S_{n-1}}\uparrow^{\S_n}
=\left(\sum\limits_{I\subseteq J} (-1)^{|J\setminus I|}
1_{\S_{\bar
J}}\uparrow^{\S_{n-1}}\right)\uparrow^{\S_n}=S^{Z_{n-1,J}}\uparrow^{\S_n},
\]
where the last equality uses Equation~\eqref{eq:altsum}. This completes the proof.
\end{proof}

\subsection{Horizontal rotations of grids}\label{Section7:cor}

In this subsection we prove that $\G_n(M_k)$ is a fine set by interpreting the corresponding grid as a horizontal rotation of a grid in $\S_{n-1}$.

\begin{lemma}\label{rotated_shuffles2_lemma2}
For every $k\ge1$,
\[
\G_n(M_k)= \G^{+^k}_{n-1}\, C_n.
\]
\end{lemma}

\begin{proof}
Recall from Remark~\ref{rem:Cn} that every $\pi\in\S_n$ can be written uniquely as $\pi=\sigma c^j$ for some $\sigma\in S_{n-1}\subset\S_n$ and $0\le j<n$.
We will show that $\pi\in\G_n(M_k)$ if and only if $\sigma\in \G^{+^k}_{n-1}$.
By Lemma~\ref{lem:GMkcdes} and Observation~\ref{obs:Gk}, this is equivalent to showing that $\cdes(\pi^{-1})\le k$ if and only if $\des(\sigma^{-1})\le k-1$.
Viewing $\sigma^{-1}$ as an element of $\S_{n}$ when considering the statistic $\cdes$, Lemma~\ref{lem:cdes_rotations} gives
$\des(\sigma^{-1})=\cdes(\sigma^{-1})-1=\cdes(\pi^{-1})-1$, and the result follows.
\end{proof}

Now we can prove the following consequences of  Theorem~\ref{thm:horizontal1}.

\begin{corollary}\label{cor:rotated_shuffles2}
For every $k\ge1$, the set $\G_n(M_k)$ is fine; namely,
$\Q(\G_n(M_k))$ is symmetric and Schur-positive.
\end{corollary}

\begin{proof}
By Lemma~\ref{rotated_shuffles2_lemma2} together with Observation~\ref{obs:Gk} and Remark~\ref{rem:Cn},
$$\G_n(M_k)= \bigsqcup_{\substack{J\subseteq[n-2] \\ |J|\le k-1}} D_{n-1,J}^{-1}C_n.$$
Since $D_{n-1,J}^{-1}C_n$ is a fine set for every $J$ by Theorem~\ref{thm:horizontal1}, it follows that $\G_n(M_k)$ is a fine set as well.
\end{proof}

\begin{corollary}\label{cor:cyc_fine}
For every $k\ge1$, the set
\[
\{\pi\in \S_n:\ \cdes(\pi^{-1})=k\}
\] is fine for the multiplicity-free sum of induced representations
$S^{Z}\uparrow^{\S_n}$ over all ribbons $Z$ of height $k$ and size $n-1$.
\end{corollary}

\begin{proof}
By Lemmas~\ref{lem:GMkcdes} and~\ref{rotated_shuffles2_lemma2}, together with Remark~\ref{rem:Cn} and Observation~\ref{obs:Gk},
\begin{multline*}\{\pi\in \S_n:\ \cdes(\pi^{-1})=k\}=\G_n(M_k)\setminus\G_n(M_{k-1})=\G^{+^k}_{n-1}C_n\setminus\G^{+^{k-1}}_{n-1}C_n
\\ =\{\pi\in\S_{n-1}:\des(\pi^{-1})=k-1\}\,C_n
=\bigsqcup_{\substack{J\subseteq[n-2] \\ |J|= k-1}} D_{n-1,J}^{-1}C_n.
\end{multline*}
By Theorem~\ref{thm:horizontal1}, $D_{n-1,J}^{-1}C_n$ is a fine set for $S^{Z_{n-1,J}}\uparrow^{\S_n}$, completing the proof.
\end{proof}

\subsection{Applications}\label{sec:horizontal_enumerative}

This subsection discusses two more consequences of Theorem~\ref{thm:horizontal1}.
The first one involves arc permutations and left-unimodal permutations in $\S_n$.
It is easy to see that $\A_n=C_n\LL_{n-1}$. Indeed, as in Remark~\ref{rem:Cn} but with $C_n$ multiplying from the left instead of from the right, every $\pi\in\S_n$ can be written uniquely as $\pi=c^k\sigma$ for $0\le k<n$ and some $\sigma\in\S_{n-1}$. In such an expression, the condition of every prefix of $\pi$ being an interval in $\bbz_n$ is equivalent to every prefix of $\sigma$ being an interval in $\bbz$. Thus, $\pi\in\A_n$ if and only if $\sigma\in\LL_{n-1}$.

On the other hand, $\A_n\ne \LL_{n-1} C_n$ for $n\ge4$. For example, $1342\notin\A_4$ is the product of $213\in\LL_3$ with $(1,2,3,4)\in C_4$.
Note also that, by Observation~\ref{obs:arc_revisited}, $\bigcup_n \A_n$ is a union of grid classes, but this is not the case for the set $\bigcup_n {\mathcal L}_{n-1}C_n$, since it is not closed under pattern containment.

Nevertheless, the following equidistribution phenomenon holds.

\begin{corollary}\label{cor:LC=CL} For every positive integer $n$,
$$\Q({\mathcal L}_{n-1}C_n)=\Q(\A_n).
$$
In particular, ${\mathcal L}_{n-1}C_n$ is set-fine.
\end{corollary}

\begin{proof}
By Equation (\ref{Eq:LL}) together with Lemma~\ref{lem:QLL} and
Theorem~\ref{thm:horizontal1}, ${\mathcal L}_{n-1}C_n$ is fine and
\[
\Q({\mathcal L}_{n-1}C_n)=\ch
\left(\sum\limits_{k=0}^{n-2}\chi^{(n-1-k,1^k)}\uparrow^{\S_n}\right)=s_{n}+s_{1^n}+\sum\limits_{k=2}^{n-2}s_{n-k,2,1^{k-2}}+2\sum\limits_{k=1}^{n-2}s_{n-k,1^k}=\Q(\A_n).
\]
The last equality follows from
Proposition~\ref{prop:rotated_leftunimodal}.
\end{proof}

\begin{corollary}\label{cor:horizontally_rotated_colayered}
For every $J\subseteq [n-2]$, if $A\subset \S_{n-1}$ is such that $\Q(A)=\Q(D_{n-1,J}^{-1})$, then $AC_n$ is a fine set for
$S^{Z_{n-1,J}}\uparrow^{\S_n}$.

In particular, for every $2\le k\le n-1$,
\[
\Q((\LLL^k_{n-1} \setminus \LLL^{k-1}_{n-1})\,C_n)=s_{n-k,1^k}+s_{n-k+1,1^{k-1}}+s_{n-k,2,1^{k-2}}.
\]
\end{corollary}

\begin{proof}
For every $0\le j<n$ and $\pi\in \S_{n-1}$, we have that $i\in\Des(\pi c^{-j})$
if either $i-j\bmod n$ belongs to $\Des(\pi)$ or $i=j$. It follows that the distribution of $\Des$ on
$\pi C_n$ is determined by $\Des(\pi)$. Thus, if $\Q(A)=\Q(D_{n-1,J}^{-1})$, then $\Q(A
C_n)=\Q(D_{n-1,J}^{-1} C_n)$. Hence, by
Theorem~\ref{thm:horizontal1}, $AC_n$ is a fine set for
$S^{Z_{n-1,J}}\uparrow^{\S_n}$.
 The second statement follows now using Proposition~\ref{prop:k-colayered_hooks}.
\end{proof}

Some conjectured generalizations of
Corollaries~\ref{cor:LC=CL}
and~\ref{cor:horizontally_rotated_colayered} will be discussed in Section~\ref{sec:open}.

\section{Other geometric operations on grids}\label{sec:other}
Aside from vertical and horizontal rotations, there are other operations on grids that preserve Schur-positivity in some circumstances. In this section we study some of them.

\subsection{Vertical and horizontal reflections}\label{sec:reflections}

Before we state our result about reflections of grids, let us
introduce a few definitions and a lemma. For a set
$D\subseteq[n-1]$, define $n-D:=\{n-i:i\in D\}$. Let
$w_0:=n(n-1)\dots 1\in\S_n$. For $\pi\in\S_n$, we have
$w_0\pi=(n+1-\pi(1))(n+1-\pi(2))\dots(n+1-\pi(n))$, $\pi
w_0=\pi(n)\dots\pi(2)\pi(1)$, and $w_0\pi
w_0=(n+1-\pi(n))\dots(n+1-\pi(2))(n+1-\pi(1))$. In particular,
\begin{equation}\label{eq:n-Des}
\Des(w_0\pi w_0)=n-\Des(\pi).
\end{equation}

\begin{lemma}\label{m3}
Let $\BBB$ be a fine set of $\S_n$. Then the statistics $\Des$ and $n-\Des$ are equidistributed over $\BBB$.
\end{lemma}

\begin{proof}
By Theorem~\ref{m1} together with 
Remark~\ref{clarification}, it suffices to prove that the statement holds when $\BBB$ is a Knuth class. Observe that for every
Knuth class ${\mathcal K}\subseteq\S_n$ of shape $\lambda$, the
set $w_0 {\mathcal K} w_0$ is also a Knuth class of shape
$\lambda$. Finally, use Equation~\eqref{eq:n-Des}.
\end{proof}

\begin{proposition}\label{prop_reflections}
Let $\G$ be a grid class, and let $\G^\mathrm{ver}$ and
$\G^\mathrm{hor}$ be the grid classes obtained by reflecting the
grid of $\G$ vertically and horizontally, respectively. If $\G_n$
is fine for the $\S_n$-representation $\rho$, then
$\G^\mathrm{ver}_n$ and $\G^\mathrm{hor}_n$ are fine for
$S^{(1^n)}\otimes \rho$.
\end{proposition}

\begin{proof}
Note that $\G^\mathrm{hor}_n=\G_n w_0$.
Since $\{w_0\}=D_{n,[n-1]}^{-1}$ is an inverse descent class, where
$S^{Z_{n,[n-1]}}\cong S^{(1^n)}$, Theorem~\ref{main} implies that
$\G_n w_0$ is a fine set for $S^{(1^n)}\otimes \rho$.

Similarly, $\G^\mathrm{ver}_n=w_0\G_n=w_0(\G_n w_0)w_0$. By Equation~\eqref{eq:n-Des}, the distribution of $\Des$ over $w_0\G_n$ equals the distribution of $n-\Des$ over $\G_n w_0$, which in turn equals the distribution of $\Des$ over $\G_n w_0$ by Lemma~\ref{m3} applied to the fine set $\G_n w_0$.
This shows that the distribution of $\Des$ is the same over $\G^\mathrm{ver}$ and $\G^\mathrm{hor}$, completing the proof.
\end{proof}

\begin{corollary}\label{cor_equid_rotation}
Let $\G$ be a Schur-positive grid class, and let $\G^\mathrm{rot}$ be the grid class obtained
by rotating the grid for $\G$ by $180$ degrees. Then
\begin{equation}\label{eq:M*}
\Q(\G_n)=\Q(\G^\mathrm{rot}_n).
\end{equation}
\end{corollary}

Note that if $\G=\G(M)$ for some $k\times l$ matrix $M$, then $\G^\mathrm{rot}=\G(M^*)$, where $M^*$ is defined by
$m_{i,j}^*= m_{k+1-i,l+1-j}$. Interestingly, Equation~\eqref{eq:M*} does not necessarily hold on
general grid classes. For example, letting $\G=\G(-1\ \ 1)$, we have
$$\Q(\G_3) = F_{3,\emptyset}+2F_{3,\{1\}}+F_{3,\{1,2\}}\ne
F_{3,\emptyset}+2F_{3,\{2\}}+F_{3,\{1,2\}}=\Q(\G^\mathrm{rot}_3).$$ In fact,
for any any non-palindromic one-row matrix $M$ one can easily
verify that $\Q(\G_n(M)) \ne \Q(\G_n(M^*))$ if $n$ is large enough.

For
$\vv=(v_1,\dots,v_k)\in \{1,-1\}^k$ let
$\vv^R:=(v_k,\dots,v_1)$.
A special case of Corollary~\ref{cor_equid_rotation} is the fact that
\[
\Q(\G^\vv_n)=\Q(\G^{\vv^R}_n).
\]

\subsection{Stacking grids}\label{sec:stacking}

Another natural geometric operation on grids consists of
stacking one grid on top of another. In the case of two
multiple-column grids, there is an ambiguity in the choice of width of the columns.
Thus, we will only consider the case where one of the
stacked grids has a single column.

\begin{defn}[The stacking operation]
For a grid class $\mathcal{H}$ and $\vv\in\{+,-\}^k$, let $\G^{\vv,
\mathcal{H}}$ ($\G^{\mathcal{H}, \vv}$) be the grid class obtained by
placing the grid for $\G^\vv$ below (atop) the one for $\mathcal{H}$.
\end{defn}

Figure~\ref{fig:above} shows an example of the stacking operation.

\begin{figure}[htb]
\centering
\begin{tikzpicture}[scale=.4]
\draw (0,0) rectangle (10,10);
\draw[dotted] (0,6)--(10,6); \draw[dotted] (5,0)--(5,6);
\draw[dotted] (0,5)--(10,5); \draw[dotted] (0,3)--(10,3); \draw[dotted] (0,1)--(10,1);
\draw[thick] (0,6)--(10,8); \draw[thick] (0,10)--(10,8);
\draw[thick] (0,5)--(5,6); \draw[thick] (0,3)--(10,5);
\draw[thick] (0,1)--(10,3); \draw[thick] (5,0)--(10,1);
\end{tikzpicture}
\caption{The grid $\G^{\G(M_3),(+,-)}$.} \label{fig:above}
\end{figure}

By Corollary~\ref{cor_equid_rotation}, for any grid class
$\G(M)$ and any $\vv\in\{+,-\}^k$ such that $\G^{\G(M), \vv}$ is fine, we have
$$
\Q(\G^{\G(M), \vv}_n)=\Q(\G^{\vv^R, \G(M^*)}_n).
$$

\begin{question}
Let $\mathcal{H}$ be a fine grid and let $\vv\in\{+,-\}^k$. Is the
stacked grid $\G^{\vv, \mathcal{H}}$ necessarily fine?
\end{question}

Computer experiments hint at an affirmative answer.
Note that $\G^{\vv, \mathcal{H}}$ is the underlying set of the multiset obtained by shuffling $\G^{\vv}$ and $\mathcal{H}$. This multiset is fine by Lemma~\ref{lemma_shuffles}.
In the special case that $\mathcal{H}$ has a single column, the grid $\G^{\vv, \mathcal{H}}$ has a single column as well, and so it is a fine set  by Corollary~\ref{one-column_zigzags}. In this section we consider some other instances when $\G^{\vv, \mathcal{H}}$ is fine.

Let $$\mathcal{J}:=\G^{\LLL^2,+}=\G\left(\begin{array}{cc}
           0 & 1 \\
           1 & 0 \\
           1 & 0 \\
           0 & 1
         \end{array}\right) \qquad \text{and} \qquad
\mathcal{K}:=\G^{-,\LLL^2}
= \G\left(\begin{array}{cc}
           1 & 0 \\
           0 & 1 \\
           -1 & 0 \\
           0 & -1
         \end{array}\right).$$
The corresponding grids are drawn in Figure~\ref{fig:HK}.
In terms of shuffles, ${\mathcal J}_n$ is the set of permutations in $\S_n$
resulting from shuffling a 2-colayered permutation with the
identity permutation, and ${\mathcal K}_n$ is the set of
permutations resulting from shuffling the reverse identity
permutation with a 2-colayered permutation.

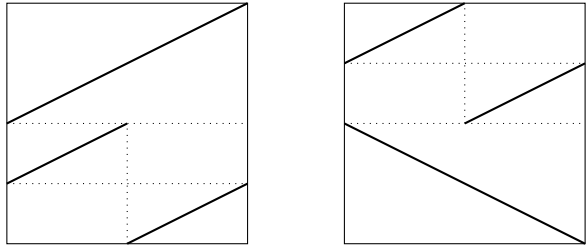
\begin{figure}[htb]
\centering
\begin{tikzpicture}[scale=.8]
\draw (0,0) rectangle (4,4); \draw[dotted] (0,2)--(4,2);
\draw[dotted] (0,1)--(4,1); \draw[dotted] (2,0)--(2,2);
\draw[thick] (0,2)--(4,4); \draw[thick] (0,1)--(2,2); \draw[thick]
(2,0)--(4,1);
\end{tikzpicture}
\hspace{10mm}
\begin{tikzpicture}[scale=.8]
\draw (0,0) rectangle (4,4); \draw[dotted] (0,2)--(4,2);
\draw[dotted] (0,3)--(4,3); \draw[dotted] (2,2)--(2,4);
\draw[thick] (0,2)--(4,0); \draw[thick] (0,3)--(2,4); \draw[thick]
(2,2)--(4,3);
\end{tikzpicture}
\caption{The grids for $\mathcal{J}$ (left) and $\mathcal{K}$
(right).} \label{fig:HK}
\end{figure}

\begin{proposition}\label{KJ:cardinality}
We have $|\mathcal{J}_n|=|\mathcal{K}_n|=(n-2)2^{n-1}+2$.
\end{proposition}

\begin{proof}
First, we prove that $|\mathcal{J}_n|=(n-2)2^{n-1}+2$.
Considering the left half of the grid for $\mathcal{J}_n$, it is clear that $\Sh_n\subset\mathcal{J}_n$. Note that
$|\Sh_n|=2^n-n$, since every $\pi\in\Sh_n$ is uniquely determined
by the choice of dots $\pi(i)$ placed on the upper and lower
segments of the grid, except for the identity permutation, which
is obtained in $n+1$ ways. It remains to count the permutations in
$\mathcal{J}_n$ that are not in $\Sh_n$.

For $\pi\in\mathcal{J}_n\setminus\Sh_n$, let $j$ be the index such that $\pi(j)=1$. Then $\pi(1)\pi(2)\dots\pi(j-1)$ is not increasing, because any drawing of $\pi$ on the grid for $\mathcal{J}$ where the entries to the left of the lowest dot are increasing would have to be a shuffle of two increasing sequences. It follows that any drawing of $\pi$ on the grid  for $\mathcal{J}$ has $\pi(j)=1$ as the leftmost dot on the lowest segment. The permutation $\st(\pi(1)\pi(2)\dots\pi(j-1))$, where $\st$ is the standardization operation defined at the end of Section~\ref{sec:grid}, can then be any permutation in $\Sh_{j-1}$ other than the identity, and so there
are $|\Sh_{j-1}|-1=2^{j-1}-j$ choices for it.
To the right $\pi(j)$, any choice for which dots are drawn on
the upper segment and which are drawn on the lower segment is
valid and gives a different permutation $\pi$.

It follows that
$$|\mathcal{J}_n\setminus\Sh_n|=\sum_{j=3}^n(2^{j-1}-j)2^{n-j}=(n-4)2^{n-1}+n+2,$$
and so $$|\mathcal{J}_n|=2^n-n+(n-4)2^{n-1}+n+2=(n-2)2^{n-1}+2.$$

\medskip

Next we prove that $|\mathcal{K}_n|=(n-2)2^{n-1}+2$ as well.
Considering the left half of the grid for $\mathcal{K}_n$, we see that $\LL_n\subset\mathcal{K}_n$. It is known that
$|\LL_n|=2^{n-1}$. It remains to count the permutations in
$\mathcal{K}_n$ that are not left-unimodal

For $\pi\in\mathcal{K}_n\setminus\LL_n$, let $j$ be the smallest index
such that $\st(\pi(1)\pi(2)\dots\pi(j))\notin\LL_j$. Equivalently, $j$ is the smallest index with the property that $\pi(j)$ is neither the largest nor the smallest of all the preceding entries. Note that
$j\ge3$, and that in any drawing of $\pi$ on the grid, the dot for
$\pi(j)$ must lie on the middle segment (the one whose left endpoint is in the center of the grid). We can assume that $\pi(j)$ is the leftmost dot on the middle segment, since the only exception would happen if $\pi(j-1)$ was on that segment as well, in which case it could be moved to the lower segment without changing the permutation.
The choices for the remaining dots are as follows.
To the left of $\pi(j)$, we can
choose a subset of $[j-1]$ to be the positions of the dots drawn
on the upper segment. The empty set and $[j-1]$ are not valid
subsets, by the choice of $j$, but all other subsets will result
in different permutations, so there are $2^{j-1}-2$ choices for
these dots. To the right of $\pi(j)$, any choice for which dots
are drawn on the middle and lower segments is valid and gives a
different permutation, so there are $2^{n-j}$ choices.

It follows that
$$|\mathcal{K}_n\setminus\LL_n|=\sum_{j=3}^n(2^{j-1}-2)2^{n-j}=(n-3)2^{n-1}+2,$$
and so $$|\mathcal{K}_n|=2^{n-1}+(n-3)2^{n-1}+2=(n-2)2^{n-1}+2.$$
\end{proof}

\begin{remark}
We do not know if the equality $|\mathcal{J}_n|=|\mathcal{K}_n|$ can be explained as part of a more general setting to produce grid classes with the same cardinality but different quasisymmetric functions.
\end{remark}

\begin{proposition}\label{J:prop} $\mathcal{J}_n$ is a fine set, and
$$\Q(\mathcal{J}_n)=s_n+\sum_{a=1}^{\lfloor\frac{n}{2}\rfloor} (n-2a+1)s_{n-a,a}+\sum_{a=1}^{\lfloor \frac{n-1}{2}\rfloor} (n-2a)s_{n-a-1,a,1}.$$
\end{proposition}

\begin{proof}
We construct a $\Des$-preserving bijection $f$ from
$\mathcal{J}_n$ to the multiset of SYT of the
shapes that index the Schur functions in the formula.

For $\pi\in\Sh_n$, let $f(\pi)$ be the recording tableau $Q$
obtained when applying RSK to $\pi$. As in Corollary~\ref{cor:QSh},
$f$ is a bijection between $\Sh_n$ and the
SYT whose shapes are given by the multiset where each two-row
partition $(n-a,a)$ has multiplicity $n-2a+1$ for each $1\le a\le
\lfloor\frac{n}{2}\rfloor$, and the shape $(n)$ has multiplicity
$1$.

For $\pi\in\mathcal{J}_n\setminus\Sh_n$, let again $j$ be such that $\pi(j)=1$.
Then $\pi(1)\pi(2)\dots\pi(j-1)$ is a shuffle (other than the
trivial increasing one) of two increasing sequences, say
$r+1,r+2,\dots,s$ and $s+1,s+2,\dots,t$, with $r<s<t$. Let $\pi'$ be the permutation obtained from $\pi$
by decreasing the entries $r+1,r+2\dots,s$ by $r-1$, and
increasing the entries $2,3,\dots,r$ by $s-r$. Then $\pi'$
consists of a non-trivial shuffle of the sequences
$2,3,\dots,s-r+1$ and $s+1,s+2,\dots,t$, followed by the entry
$\pi'(j)=1$, followed by a shuffle of the sequence
$s-r+2,s-r+3,\dots,s$ and $t+1,t+2,\dots,n$. Equivalently,
$\pi'$ is an arbitrary element of $\Sh_{n-1}\setminus\{12\dots
(n-1)\}$ with its entries increased by $1$ and the letter $1$
inserted anywhere after the first descent. Note that $\Des(\pi')=\Des(\pi)$.

To describe $f(\pi)$, first apply RSK to the permutation $\pi'$
with the entry $\pi'(j)=1$ removed, and let $Q'$ be the recording
tableau with entries $[n]\setminus\{j\}$. Let $f(\pi)$ be the
tableau obtained from $Q'$ by adding a third row consisting of a
box with the entry $j$.

By construction,
$\Des(f(\pi))=\Des(Q')\cup\{j-1\}=\Des(\pi')=\Des(\pi)$.
Recall from Corollary~\ref{cor:QSh} that taking the recording tableau under RSK gives a bijection between $\Sh_{n-1}\setminus\{12\dots (n-1)\}$ and the
multiset of SYT where each tableau of shape $(n-a-1,a)$ appears $n-2a$ times for $1\le a\le \lfloor\frac{n-1}{2}\rfloor$. It follows that the map $f$ is a bijection between
$\mathcal{J}_n\setminus\Sh_n$ and the multiset of SYT with $3$
rows and exactly one box on the third row, where every tableaux of
shape $(n-a-1,a,1)$ appears $n-2a$ times for each $1\le a\le
\lfloor\frac{n-1}{2}\rfloor$.
\end{proof}

\begin{proposition}\label{K:prop} $\mathcal{K}_n$ is a fine set, and
$$\Q(\mathcal{K}_n)=s_n+s_{1^n}+2\sum_{k=1}^{n-2} s_{n-k,1^k}+2\sum_{k=1}^{n-3} s_{n-k-1,2,1^{k-1}}.$$
\end{proposition}

\begin{proof}
Again, we construct a $\Des$-preserving bijection $g$ from
$\mathcal{K}_n$ to the appropriate multiset of SYT.

For $\pi\in\LL_n$, let $g(\pi)$ be the recording tableau $Q$
obtained when applying RSK to $\pi$. As in Lemma~\ref{lem:QLL},
$g$ is a bijection between $\LL_n$ and the set of SYT whose shape is a hook,
and the contribution of these permutations to the quasisymmetric function is
$\sum_{k=1}^{n-1} s_{n-k,1^k}$.

For $\pi\in\mathcal{K}_n\setminus\LL_n$, let again $j$ be the
smallest such that $\st(\pi(1)\pi(2)\dots\pi(j))\notin\LL_j$, and
consider the drawing of $\pi$ on the grid where the dot for
$\pi(j)$ is the leftmost dot on the right segment. As shown in the
proof of Proposition~\ref{KJ:cardinality}, placing $\pi(j)$ there uniquely determines the
segments on the grid in which each of the other dots lie. Consider
three cases:

\begin{itemize}
\item If $\pi(1)\pi(2)\dots\pi(j-1)$ is increasing, which means
that the dot for $\pi(1)$ is the only one among these in the lower
segment, let $g(\pi)$ be the hook whose entries in the first
column are the indices $i$ such that $\pi(i)\le\pi(j)$. This
construction is a bijection between permutations in this case and
hooks with at least two rows and with the entry $2$ in the first row.

\item If $\pi(1)\pi(2)\dots\pi(j-1)$ is decreasing, which means
that the dot for $\pi(1)$ is the only one among these in the upper
segment, let $g(\pi)$ be the hook whose entries in the first row
are the indices $i$ such that $\pi(i)\ge\pi(j)$. This construction
is a bijection between permutations in this case and hooks with at least two columns and with the entry $2$ in the first column.

\item Finally, if $\pi(1)\pi(2)\dots\pi(j-1)$ is neither
increasing nor decreasing, let $g(\pi)$ be the SYT whose first
$j-1$ entries are given by the recording tableau of
$\pi(1)\pi(2)\dots\pi(j-1)$, the entry $j$ is in the box in the
second row and second column, and the remaining values $i>j$ are
appended to the first row if $\pi(i)>\pi(j)$ and to the first column
if $\pi(i)<\pi(j)$. This construction gives a 2-to-1 map between
permutations in this case and SYT of shape consisting of a hook
plus a box. Indeed, placing the dot $\pi(1)$ in the top or the
bottom segment of the grid results in the same tableau $f(\pi)$
but in a different permutation $\pi$.
\end{itemize}

The
contribution of the permutations where $\pi(1)\pi(2)\dots\pi(j-1)$
is increasing or decreasing is $\sum_{k=1}^{n-2} s_{n-k,1^k}$.
 The contribution of the remaining
permutations is $2\sum_{k=1}^{n-3} s_{n-k-1,2,1^{k-1}}$.
In each of the three cases it easy to check that
$\Des(g(\pi))=\Des(\pi)$.
\end{proof}

\begin{remark}\label{orb_shuffles}
It follows from Proposition~\ref{J:prop} together with Young's
rule~\cite[Prop. 7.18.7]{Stanley_ECII} that $\Q(\mathcal{J}_n)$ is
given by an alternating sum of permutation representations:
\[
\Q(\mathcal{J}_n)=\ch\left(2\sum\limits_{a=1}^{\lfloor
\frac{n}{2}\rfloor -1}
\left(1\uparrow_{\S_{n-a-1,a,1}}^{\S_n}-1\uparrow_{\S_{n-a-1,a+1}}^{\S_n}\right)+1\uparrow_{\S_{\lfloor\frac{n}{2}\rfloor,\lfloor\frac{n}{2}\rfloor,n
\bmod 2}}^{\S_n} \right).
\]

Similarly, letting $V_i$ be an $i$-dimensional vector space over
$\CC$
and recalling that the exterior algebra $\wedge(V_i)$  is
isomorphic, as an $\S_i$-module,  to the multiplicity-free sum of
all hooks of size $i$~\cite[Ex. 4.6]{Fulton_Harris_book},
Proposition~\ref{K:prop} implies that
\[
\Q(\mathcal{K}_n)=\ch
\left(2\left(\wedge(V_{n-1})\uparrow_{\S_{n-1}}^{\S_n}-\wedge(V_n)\right)+\chi^{(n)}+\chi^{(1^n)}\right).
\]

Computing the corresponding character degrees provides an
alternative proof to Proposition~\ref{KJ:cardinality}.
\end{remark}

\section{Explicit multiplication of one-column grid classes}\label{prod_one_col}

Multiset and set products of one-column grid classes are fine, as can bee seen using
Propositions~\ref{one-column_descent_classes} and~\ref{conj_des_classes} and also the
fact, which appears in the proof of the latter proposition, that set products of inverse descent classes are unions of inverse descent classes.
 Furthermore, it follows from
Proposition~\ref{one-column_descent_classes} and
Corollary~\ref{Solomon_basis} that the multiset product of
one-column grid classes in $\S_n$ is spanned by one-column grid
classes with $\vv\in \{+,-\}^{n-1}$, with nonnegative constants,
which are essentially the structure constants of the Solomon
descent algebra. In this section we give an explicit compact
description of the set product of one-column grid classes and,
more generally, of the set product of a one-column grid class with
an arbitrary grid class.

Given a grid class $\G(M)$, let $M^c$ be the matrix obtained by flipping
$M$ upside down and changing the signs of its entries. Then $\G(M^c)=\{w_0\pi:\pi\in \G(M)\}=\G(M)^\mathrm{ver}$,
the grid class obtained from $\G(M)$ by vertical reflection.

\begin{proposition}\label{prop:product_onecolumnA}
Let $\vv\in\{+,-\}^r$ and $A=\G(M)$ be a grid class. Then
$\{\G^\vv A\}$ is the grid class whose matrix is obtained by
stacking copies of $M$ and $M^c$ on top of each other, where the
$i$th copy from the bottom is $M$ if $v_i=+$ and $M^c$ if
$v_i=-$.
\end{proposition}

An example of a product of the form $\{\G^\vv A\}$ is given in
Figure~\ref{fig:product_onecolumn}. The dots represent the
permutations $\pi=4532617$, $\sigma=4176235$, and their product
$\pi\sigma=2471536$.

\begin{figure}[htb]
\centering
\begin{tikzpicture}[scale=.6]
\draw (0,0) rectangle (6,6); \draw[thick] (6,0)--(0,3)--(6,6);
\draw (7,3) node {$\times$}; \draw (8,0) rectangle (14,6);
\draw[dotted] (8,4)--(14,4); \draw[thick] (6,0)--(0,3)--(6,6);
\draw[thick] (8,0)--(14,2)--(8,4); \draw[thick] (8,6)--(14,4);
\draw (15,3) node {$=$}; \draw (16,0) rectangle (22,6);
\draw[dotted] (16,1)--(22,1); \draw[dotted] (16,5)--(22,5);
\draw[thick] (16,0)--(22,1); \draw[thick] (16,1)--(22,2)--(16,3
)--(22,4)--(16,5); \draw[thick] (16,6)--(22,5);

\def\xa{6/7}
\def\xb{9/7}
\def\xc{15/7}
\def\xd{21/7}
\def\xe{27/7}
\def\xf{33/7}
\def\xg{39/7}
\def\ya{3-\xa/2}
\def\yc{3-\xc/2}
\def\yd{3-\xd/2}
\def\yf{3-\xf/2}
\def\yb{3+\xb/2}
\def\ye{3+\xe/2}
\def\yg{3+\xg/2}
\draw[fill,red] (\xa,\ya) circle (0.1); \draw[fill,blue] (\xb,\yb)
circle (0.1); \draw[fill,red] (\xc,\yc) circle (0.1);
\draw[fill,red] (\xd,\yd) circle (0.1); \draw[fill,blue] (\xe,\ye)
circle (0.1); \draw[fill,red] (\xf,\yf) circle (0.1);
\draw[fill,blue] (\xg,\yg) circle (0.1);

\def\yb{\xb/3}
\def\ye{\xe/3}
\def\ya{4-\xa/3}
\def\yf{4-\xf/3}
\def\yc{6-\xc/3}
\def\yd{6-\xd/3}
\def\yg{6-\xg/3}
\draw[fill,red] (\xa+8,\ya) circle (0.1); \draw[fill,red]
(\xb+8,\yb) circle (0.1); \draw[fill,blue] (\xc+8,\yc) circle
(0.1); \draw[fill,red] (\xd+8,\yd) circle (0.1); \draw[fill,blue]
(\xe+8,\ye) circle (0.1); \draw[fill,red] (\xf+8,\yf) circle
(0.1); \draw[fill,blue] (\xg+8,\yg) circle (0.1);

\def\ya{1+\xa/6}
\def\yb{3-\xb/6}
\def\yc{6-\xc/6}
\def\yd{\xd/6}
\def\ye{3+\xe/6}
\def\yf{1+\xf/6}
\def\yg{6-\xg/6}
\draw[fill,red] (\xa+16,\ya) circle (0.1); \draw[fill,red]
(\xb+16,\yb) circle (0.1); \draw[fill,blue] (\xc+16,\yc) circle
(0.1); \draw[fill,red] (\xd+16,\yd) circle (0.1); \draw[fill,blue]
(\xe+16,\ye) circle (0.1); \draw[fill,red] (\xf+16,\yf) circle
(0.1); \draw[fill,blue] (\xg+16,\yg) circle (0.1);

\end{tikzpicture}
\caption{The product of one-column grid classes.}
\label{fig:product_onecolumn}
\end{figure}
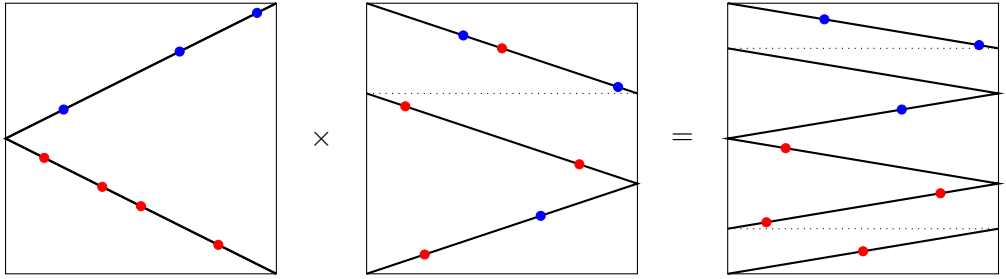

\begin{proof}[Proof of Proposition~\ref{prop:product_onecolumnA}]
Let $M^\vv$ be the matrix described in the statement, obtained by
stacking copies of $M$ and $M^c$ on top of each other according to
the vector $\vv$. The $k$th copy of the grid $\G(M)$ or $\G(M^c)$
in $\G(M^\vv)$ starting from the bottom will be called the $k$th
{\em block} of $\G(M^\vv)$. Similarly, when referring to the $j$th
segment of the grid $\G^\vv$, we consider that the segments are
ordered from bottom to top. Our goal is to show that
$\{\G^\vv\G(M)\}=\G(M^\vv)$.

To prove the inclusion $\{\G^\vv \G(M)\}\subseteq\G(M^\vv)$, let
$\pi\in\G^\vv$ and $\sigma\in \G(M)$, and consider drawings of
these permutations on their corresponding grids. Let us show how
to draw the permutation $\pi\sigma$ on the grid $\G(M^\vv)$. For
every $1\le k\le r$, let $J_k$ be the set of indices $j$ such that
the dot for $\pi(j)$ is on the $k$th segment of the grid
$\G^\vv$. Consider the dots on the grid $\G(M)$ corresponding to entries
$\sigma(i)$  such that $\sigma(i)\in J_k$, and copy these dots to the
$k$th block of the grid $\G(M^\vv)$ in the same $x$-coordinates
that they had in the grid $\G(M)$. If $v_k=-$, this block is a
copy of the grid $\G(M^c)$, in which case the dots are flipped
accordingly. Doing this for every $k$ results in a drawing of the
permutation $\pi\sigma$ on the grid $\G(M^\vv)$. Indeed, let us show that
if $\pi\sigma(i)>\pi\sigma(i')$ then the dot corresponding to $\pi\sigma(i)$ is placed higher than the dot for $\pi\sigma(i')$.
If $\sigma(i)$ and $\sigma(i')$ belong to the same set $J_k$, then $\pi\sigma(i)>\pi\sigma(i')$ is equivalent to $\sigma(i)>\sigma(i')$ if $v_k=+$ (in which case our drawing preserves the relative height of the dots
$\sigma(i)$ and $\sigma(i')$ in $\G(M)$), and to $\sigma(i)<\sigma(i')$ if $v_k=-$ (in which case our drawing reverses the relative height).
If $\sigma(i)$ and $\sigma(i')$ belong to different sets $J_k$ and $J_{k'}$, respectively, then $\pi(\sigma(i))>\pi(\sigma(i'))$ implies that
$k>k'$, and our drawing places $\pi(\sigma(i))$ higher than $\pi(\sigma(i'))$.

Next we show that $\G(M^\vv)\subseteq\{\G^\vv \G(M)\}$. Given a
drawing of a permutation $\tau$ on the grid $\G(M^\vv)$, let
$\sigma$ be the permutation whose drawing is obtained by
overlaying the $r$ blocks which are copies of $\G(M)$ or
$\G(M^c)$, flipping those corresponding to $\G(M^c)$. Then we
construct $\pi$ as follows. For every $i$, if the dot for
$\tau(i)$ is in the $k$th block in the grid $\G(M^\vv)$, draw the dot
for $\pi(j)$, where $j=\sigma(i)$, in the $k$th segment of the grid
$\G^\vv$. One can check that the resulting permutations $\sigma$
and $\pi$ satisfy $\pi\sigma=\tau$.
\end{proof}

One-column grid classes are, up to taking the inverse of its
elements, equivalent to the so-called {\em $W$-sets} considered in
\cite{AMR,AAA}, which also appear in~\cite{AE} under the name
$\sigma$-classes. Compositions (i.e. products) of W-sets were
studied in \cite{AAA}. The special case of
Proposition~\ref{prop:product_onecolumnA} where $A$ is a
one-column grid class $\G^\ww$, stated as
Corollary~\ref{cor:product_onecolumn} below, is equivalent to
\cite[Theorem 6]{AAA}. It gives an explicit description of the set
product of one-column grid classes.

For a vector  $\ww=(w_1,w_2,\dots,w_s)\in\{1,-1\}^s$, we use the
notation $\ww^1=\ww$ and $\ww^{-1}=(-w_s,\dots,-w_2,-w_1)$. Given
$\vv=(v_1,\dots,v_r)\in\{1,-1\}^r$ and $\ww\in\{1,-1\}^s$, define
the product $\vv\star\ww:=(\ww^{v_1},\ww^{v_2},\dots,\ww^{v_r})$.
For example, $(-1,1)\star(1,-1,-1)=(1,1,-1,1,-1,-1)$.

\begin{corollary}[\cite{AAA}]\label{cor:product_onecolumn}
For every $\vv\in\{1,-1\}^r$ and $\ww\in\{1,-1\}^s$,
$$\{\G^\vv\G^\ww\}=\G^{\vv\star\ww}.$$
\end{corollary}

Figure~\ref{fig:product_onecolumn} shows how
$\G^{-+}\G^{+--}=\G^{++-+--}$ as sets.

\begin{remark} Even if $A$ is a fine grid class, it is not necessarily true that $\{\G^\vv A\}$ is a fine set. For example, if $A=\LLL^2$ and $\vv=-+$, we have that
$$\{\G^\vv A\}=\{\LL \LLL^2\}=\G\left(\begin{array}{cc}
           1 & 0 \\
           0 & 1 \\
           0 & -1 \\
           -1 & 0
         \end{array}\right),$$
which is not a fine set (or even symmetric) already for $n=6$.

On the other hand, the question of whether $\G^\vv A$ is always a
fine multiset remains open. An affirmative answer would follow
from Conjecture~\ref{conj:desfine}.
\end{remark}

\section{Final remarks and open problems}\label{sec:open}

We have shown in Corollary~\ref{cor:rotated_shuffles2} that vertically
rotated one-column grid classes are fine when all slopes
have the same sign, that is, $\{C_n\G_n^\vv\}$ is a fine set when
$\vv=+^k$ or $\vv=-^k$ (the latter case follows by symmetry using Propositon~\ref{prop_reflections}). By
Propositions~\ref{prop:rotated_leftunimodal}
and~\ref{prop_reflections}, this phenomenon also holds when
$\vv=-+$ or $\vv=+-$. Computer experiments suggest that the
following more general statement is true.

\begin{conjecture}\label{conj:rotations-onecolumn}
For every one-column grid class $\G^\vv$, the set $\{C_n\G_n^\vv\}$ is a fine set.
\end{conjecture}

The following conjecture suggests a far-reaching generalization of
Corollary~\ref{cor:LC=CL}. Recall that one can interpret $D_{n-1,J}^{-1}$ as a subset of $\S_n$
consisting of permutations that fix $n$.

\begin{conjecture}\label{Conj_CD_n-1}
For every  $J\subseteq [n-2]$,
\[
\Q(C_n D_{n-1,J}^{-1})=\Q(D_{n-1,J}^{-1}C_n).
\]
Thus, by Theorem~\ref{thm:horizontal1}, $C_n D_{n-1,J}^{-1}$ is
a fine set for $S^{Z_{n-1,J}}\uparrow^{\S_n}$.
\end{conjecture}

Note that
both $C_n D_{n-1,J}^{-1}$ are $D_{n-1,J}^{-1}C_n$ are sets (that
is, elements have multiplicity one).
Conjecture~\ref{Conj_CD_n-1} no longer holds when
$C_n$ is replaced by a general fine set $\BBB\subset \S_n$.

\medskip

The following generalization of
Theorem~\ref{thm:horizontal1} and Corollary~\ref{cor:horizontally_rotated_colayered} will be proved in a forthcoming
paper~\cite{ER_new}.

\begin{theorem}\label{thm:horizontal_induction}
If $\BBB\subseteq \S_{n-1}$ is a fine set for the
$\S_{n-1}$-representation $\rho$, then $\BBB C_n$ is a fine set for
$\rho\uparrow^{\S_n}$.
\end{theorem}

It should be noted that an analogous statement for vertical
rotation does not hold. For example,
$\{2143, 2413\}$ is a Knuth class in $\S_4$, thus fine, but
$C_5\{2143, 2413\}$ is not fine.

\medskip

Regarding multiset products of fine sets and inverse descent classes, computer experiments support the
following conjecture.

\begin{conjecture}\label{conj:desfine}
Let $\BBB\subset \S_n$ be fine. Then, for every $J\subseteq
[n-1]$,
\[
\Q(D_{n,J}^{-1} \BBB)=\Q(\BBB D_{n,J}^{-1}).
\]
In particular, by Theorem~\ref{main_FD}, the multiset
$D_{n,J}^{-1} \BBB$ is fine.
\end{conjecture}

Note that when $\BBB=C_n$, the conjecture
involves horizontal and vertical rotations of an inverse descent
class.

\medskip

Given a pair of Knuth classes $A$ and $B$  in $\S_n$ of shapes
$\lambda$ and $\mu$, in many cases the multiset $AB$ is fine for
the tensor product $\chi^\lambda\otimes \chi^\mu$. Equivalently,
we have that $\Q(AB)=\Q(A)*\Q(B)$, where $*$ denotes the Kronecker
product. For example, consider the Knuth class
$A=\{21435,21453,24135,24153,24513\}\subseteq \S_5$, which
satisfies $\Q(A)=s_{3,2}$. Then $A^2=\{A^2\}$ and
$$
\Q(A^2)=s_5+s_{4, 1}+s_{3, 2}+s_{3, 1, 1}+s_{2, 2, 1}+s_{2, 1, 1,
1}=\ch(\chi^{3,2}\otimes \chi^{3,2})=s_{3,2}*s_{3,2}.
$$
The identity $\Q(AB)=\Q(A)*\Q(B)$ holds for all Knuth classes $A$
and $B$ in $\S_4$, but not in general. For
$n\ge5$, there are pairs of Knuth classes $A,B\in\S_n$ for which
$\Q(AB)$ is not even symmetric.

\begin{question} For which pairs of Knuth classes are their set and multiset products fine?
\end{question}

\medskip

We conclude with a natural question regarding restriction of
Schur-positive grids.

\begin{question}
Let $\G$ be a grid class. Does $\Q(\G_n)$ being Schur-positive
imply that $\Q(\G_{n-1})$ is Schur-positive?
\end{question}

\medskip

\noindent{\bf Acknowledgements.} The authors thank Ron Adin,
Michael Albert, Christos Athanasiadis, Mike Atkinson, Zach Hamaker and Bruce Sagan for useful discussions,
comments and references. The authors also thank two anonymous referees for their thorough comments that have improved the presentation.
The first author was partially supported by grant \#280575 from the Simons Foundation and by grant H98230-14-1-0125 from the NSA.
The second author was partially supported by Dartmouth's Shapiro visitors fund.

\end{document}